%% file: main.tex
\documentclass[11pt]{article}

\input{preamble.tex}

\input{macros.tex}

\usepackage{authblk}

\title{When Rates Are Geometric: \\Rate-Certificate Transfer for Contact Splittings in Optimization}
\author[1]{George A. Kevrekidis}
\date{July 2026}

\affil[1]{Department of Applied Mathematics and Statistics, Johns Hopkins University, Baltimore, MD 21218, USA}

\begin{document}

\maketitle

\begin{abstract}
Discrete optimization algorithms are often analyzed through continuous-time limiting ODEs, but a convergence certificate for the limiting ODE is not automatically a certificate for the discrete algorithm. This paper develops contact Hamiltonian systems as a setting in which the transfer can be made precise. A contact Hamiltonian $H$ on $J^1(\R^n)$ obeys the intrinsic decay identity $\dot H = -H\,\partial_s H$, so an augmented energy $\mcE$, defined as a functional of $H$, together with the conformal rate $\partial_s H$, is a continuous-time rate certificate whenever $\mcE$ controls the objective gap. Our main theorem states, under three named and independently checkable hypotheses, that an order-$r$ contact splitting with step size $h$ transfers this certificate over the finite horizon controlled by backward error analysis. The discrete decay envelope is governed by the modified conformal factor up to $O(h^r)$ perturbations plus a backward-error shadowing defect, and the decay mechanism is inherited exactly because the modified Hamiltonian is itself a contact Hamiltonian. We use quadratic heavy ball as a fully solvable analytical example of this contact-certificate mechanism. Its projected dissipative-leapfrog spectrum agrees with established conformal-symplectic optimization theory, while the augmented contact Hamiltonian yields a sharp objective-to-certificate comparison that verifies the transfer hypotheses. For strongly convex objectives with state-dependent damping, an explicit Bregman-type Lyapunov certificate instead transfers by a separate auxiliary-shadowing corollary. The decomposition $H=K+V+D$ into a kinetic term, a potential that encodes the objective, and a dissipation term then serves as a design template with a catalogue of closed-form kinetic and dissipation sub-flows, including contact-specific damping families. Numerical experiments confirm the predicted conformal-factor tracking orders and show competitive performance on ill-conditioned benchmarks and deep-learning tasks.
\end{abstract}

\section{Introduction}
\label{sec:intro}
Accelerated first-order methods are a central object in optimization, from Polyak's heavy-ball method to Nesterov's accelerated gradient method~\cite{polyak1964heavyball,nesterov1983method}. A major line of work explains these algorithms through continuous-time dynamical systems: Nesterov acceleration admits a limiting ODE with a transparent Lyapunov certificate~\cite{su2016differential}, while variational, Bregman, mirror-descent, and inertial formulations reveal broader families of accelerated flows~\cite{wibisono2016variational,krichene2015accelerated,attouch2018fast,wilson2021lyapunov,muehlebach2019dynamical}. These perspectives have clarified the origin of acceleration and supplied powerful design principles for optimization algorithms.

At the same time, a convergence certificate for a limiting ODE is not automatically a convergence certificate for the discrete algorithm. A discretization can perturb the effective damping, modify the Lyapunov functional, or expose higher-order terms that are invisible in the formal continuous-time limit~\cite{shi2022understanding}. Thus the relevant question is not only which continuous dynamics have desirable rates, but which discrete maps preserve the geometric or Lyapunov structure responsible for those rates.

This paper shows that \emph{contact Hamiltonians provide a continuous-time rate certificate, and contact splittings transfer that certificate to discrete time over the finite horizon controlled by backward error analysis.} The certificate is the pair $(\mcE,\partial_s H)$: $\mcE$ is an augmented Lyapunov quantity controlling the objective gap (formalized below as a shifted Hamiltonian), while the conformal factor $\partial_s H$ supplies the decay rate through the intrinsic contact identity $\dot H=-H\,\partial_s H$. Our algorithmic constructions are obtained by choosing a kinetic term $K$, the potential $V=f$ given by the objective $f:\R^n\to\R$, and a dissipation term $D$, and then discretizing the resulting autonomous contact system with a structure-preserving splitting whose sub-flows are exact contactomorphisms in closed form. The same scalar quantity $\partial_s H$ thus determines the decay envelope in continuous time and, under contact splitting, its modified version governs the discrete rate transfer. The theory is developed for \emph{autonomous} contact Hamiltonians $H(x,p,s)$. Classical time-dependent damping laws, such as the Nesterov ODE with damping proportional to $1/t$, enter the framework through an autonomous lift discussed in \cref{sec:recovery}.

\subsection{Related Work}
Continuous-time approaches to optimization have produced a detailed dictionary between accelerated algorithms, dissipative ODEs, and Lyapunov functionals~\cite{su2016differential,wibisono2016variational,wilson2021lyapunov,muehlebach2019dynamical,jordan2018dynamical}, with control-theoretic dissipativity arguments providing a complementary route to the same rate certificates~\cite{hu2017dissipativity}. This dictionary is especially useful for design, but its discrete consequences depend on the integrator, since preserving the rate certificate requires more than matching the vector field to first order. This motivates structure-preserving discretizations, where backward error analysis transports continuous invariants or modified invariants to the numerical map~\cite{hairer2006geometric,mclachlan2002splitting}.

Symplectic and variational discretizations of accelerated flows have been advocated as a route to rate-preserving algorithms~\cite{betancourt2018symplectic,duruisseaux2021adaptive}, and dissipative Hamiltonian dynamics driven by stochastic gradients has a parallel history on the sampling side~\cite{chen2014sghmc}. Standard conformal-symplectic optimization treats phase-space-independent contraction rates, while presymplectic and symplectified approaches obtain backward-error tools through an enlarged conservative system~\cite{francca2021dissipative,francca2020conformal}. Nonseparable Hamiltonians may additionally require devices such as phase-space doubling. These constructions provide important precedents for rate transfer, but do not intrinsically represent state-dependent conformal factors on the original phase space. Fran\c{c}a et al.\ analyze the dissipative leapfrog on quadratic objectives, including its exact conformal contraction and stability properties. The projected $(x,p)$ map in our quadratic example is an adjoint ordering of that integrator and has the same characteristic polynomial. We revisit it to demonstrate how the augmented contact Hamiltonian supplies the certificate required by our transfer theorem. Because ordinary symplectic systems are conservative, dissipation must be introduced through time dependence, conformal symplecticity, lifting, or a contact formulation~\cite{maddison2018hamiltonian_descent}.

Contact Hamiltonian systems provide an intrinsic odd-dimensional geometry for dissipation, contact splitting and variational integrators have already been studied in the geometric integration literature~\cite{vermeeren2019contact,bravetti2020numerical}, and contact-geometric descriptions have found use in adjacent statistical settings~\cite{goto2019contact_MC}. In optimization, contact transformations have been used to reinterpret Bregman dynamics and accelerated flows~\cite{bravetti2023bregman}, with a focus on making the hard-to-integrate Bregman Hamiltonian separable via a contact change of variables. We instead use the \textit{contact decay identity itself as a rate certificate}, show that contact backward error analysis transfers this certificate through the modified conformal factor of an order-$r$ contact splitting, and use the resulting $H = K + V + D$ decomposition, defined in \cref{sec:master}, as a design template for a tractable subclass of contact optimizers, not as a claim that all useful contact Hamiltonians are naturally separable in these variables. The splitting infrastructure we rely on (exact-contact sub-flows generated by strict contactomorphisms and prolonged diffeomorphisms, their contact Baker--Campbell--Hausdorff calculus, and the local universality of the Lie algebra they generate) is developed in~\cite{kevrekidis2026local}. The present paper specializes that machinery to optimization and adds the certificate interpretation.

\subsection{Contributions and Structure}
We begin with the contact-Hamiltonian formalism and the closed-form sub-flow catalogue (\cref{sec:contact_opti}), state the certificate-transfer theory under named assumptions (\cref{sec:main}), analyze representative optimizer families in closed form (\cref{sec:closed_form}), and report numerical experiments (\cref{sec:computation}).

The main contributions of this paper are:
\begin{itemize}
    \item \textbf{A conformal-rate certificate and its discrete transfer.} We identify the pair $(\mcE,\partial_s H)$ induced by a contact Hamiltonian as a continuous-time rate certificate, and prove a conditional transfer theorem (\cref{thm:rate}): under three named, independently checkable hypotheses (\cref{ass:reg,ass:cert,ass:bea}), an order-$r$ contact splitting with step size $h$ transfers the decay envelope up to $O(h^r)$ perturbations of the conformal factor plus a backward-error shadowing defect.
    \item \textbf{Closed-form certificate examples.} We use quadratic heavy ball, whose dissipative-leapfrog spectrum was analyzed in the conformal-symplectic setting~\cite{francca2020conformal,francca2021dissipative}, as a fully solvable test of the augmented contact certificate. The contact certificate recovers the known exact spectral contraction of the projected Strang map from a different starting point, together with a sharp closed-form constant comparing the objective gap to the certificate (\cref{prop:quadratic}). The contact formulation is the more general of the two, since it is not restricted to constant damping, so agreeing exactly on this example is the consistency check one wants. For strongly convex objectives with state-dependent damping we give an explicit Bregman-type Lyapunov certificate that transfers by auxiliary shadowing (\cref{lem:bregman_lyap,cor:aux_transfer}), and the cumulative discrete conformal factor of the stated master splitting is an \emph{exact} quadrature of the damping along the numerical trajectory (\cref{prop:exact_conformal}).
    \item \textbf{A design template.} The decomposition $H = K + V + D$, with a catalogue of closed-form exact-contact sub-flows, recovers standard dissipative models and introduces contact-specific state-dependent, nonlinear, and momentum-coupled damping families. Experiments confirm the predicted second- and fourth-order conformal-factor tracking and competitive performance on ill-conditioned benchmarks and deep-learning tasks.
\end{itemize}

\section{Contact-Hamiltonian Optimization}
\label{sec:contact_opti}
We provide a minimal introduction to the contact-Hamiltonian formalism before discussing its application to optimization. Throughout this paper, we work on the extended phase space $J^1(\R^n)\simeq T^*\R^n \times \mathbb{R}$, with $T^*\R^n$ the cotangent bundle of $\R^n$, in coordinates $(x,p,s)$, where $x\in\R^n$ are the model parameters we wish to optimize, $p\in\R^n$ are the associated momenta, and $s\in\R$ is an auxiliary variable that tracks dissipation or acceleration in the standard contact-Hamiltonian description of non-conservative dynamics \cite{de2019contact,bravetti2023bregman}. We use the canonical contact form $\alpha = ds - p^T dx$. Sign conventions for contact Hamiltonian systems vary in the literature, and all rate statements below are tied to this choice (see \cref{app:contact_dynamics}).

\subsection{Contact Geometry}
\paragraph{Contact Flows.} A contact Hamiltonian $H: T^*\R^n \times \mathbb{R} \to \mathbb{R}$ is a smooth function that generates dynamics on the extended phase space via the associated vector field $X_H$:
\begin{equation}\label{eq:contact_eom}
    \dot{x} = \nabla_p H\qc \dot{p} = -\nabla_x H - p \frac{\partial H}{\partial s}\qc \dot{s} = p^T \nabla_p H - H.
\end{equation}
The associated flow $\Phi_H^t$ is a one-parameter family of contact-diffeomorphisms that maps initial conditions to their time-$t$ evolution under $X_H$. It is easy to verify that if $H$ is independent of $s$, then one recovers the familiar symplectic Hamiltonian dynamics for $(x,p)$, while $s$ evolves according to $\dot{s} = p^T \nabla_p H - H$. Thus, contact Hamiltonian systems can be viewed as a generalization of symplectic Hamiltonian systems that allow for \textit{non-conservative dynamics}, which we will leverage for optimization. If $H$ depends on $s$, the contact Hamiltonian is not conserved along the flow, but instead evolves according to the following identity:
\begin{equation}\label{eq:Hdot}
    \frac{\dx}{\dx t} H = -H \frac{\partial H}{\partial s}.
\end{equation}
Integrating~\eqref{eq:Hdot} along the flow $\Phi_H^t$ from an initial point $z=(x,p,s)\in J^1(\R^n)$ gives the continuous decay envelope
\begin{equation}\label{eq:H_envelope}
    H\bigl(\Phi_H^t z\bigr)
    = H(z)\exp(-\int_0^t
    \frac{\partial H}{\partial s}\bigl(\Phi_H^\tau z\bigr)\,d\tau),
\end{equation}
so $\partial_s H$ acts as an instantaneous geometric decay rate. We record this identity here as the basic contact analogue of energy conservation. Its rate interpretation is the subject of \cref{sec:main}.

\paragraph{Conformal Factor.} A discrete contact map $\psi: T^*\R^n \times \mathbb{R} \to T^*\R^n \times \mathbb{R}$ is a smooth map that preserves the contact structure, i.e. $\psi^* \alpha = e^{\lambda} \alpha$ for some scalar function $\lambda(x,p,s)$, where $\alpha = ds - p^T dx$ is the canonical contact form on $J^1(\R^n)$. We refer to $e^{\lambda}$ as the conformal factor associated with $\psi$. If the map $\psi$ is generated by a contact Hamiltonian $H$ via $\psi = \Phi_H^h$ for some step size $h>0$, then we have the following discrete analogue of the continuous-time identity for the evolution of $H$:
\begin{equation}
    H \circ \psi = e^{-\int_0^h \frac{\partial H}{\partial s}(\Phi_H^t) dt} H.
\end{equation}

\paragraph{Discretization.} The main analytical tool we use to transfer results from the continuous-time setting to the discrete-time setting is backward error analysis (BEA)~\cite{hairer2006geometric} applied to splitting integrators for a particular choice of contact Hamiltonian. Given a contact Hamiltonian $H$, an integrator is a discrete contact map $\psi$ that approximates the flow $\Phi_H^h$ for some step size $h>0$, i.e.
\begin{equation}
    \psi = \Phi_H^h + \mathcal{O}(h^{r+1}),
\end{equation}
where $r$ is the order of the integrator and equality is taken in a $C^k(U)$ topology for some $k\geq 0$ and over a compact subset $U\subset J^1(\R^n)$. If a Hamiltonian is expressed as a sum of simpler Hamiltonians, $H = \sum_{i=1}^m H_i$, a splitting integrator is a composition of the flows generated by the simpler Hamiltonians, e.g. $\psi = \Phi_{H_m}^{a_m h} \circ \cdots \circ \Phi_{H_1}^{a_1 h}$ for some coefficients $a_i$, with more general compositions revisiting the same sub-flow over several sweeps as in~\eqref{eq:strang}. Any such composition of exact contact sub-flows is a contact map, and for sufficiently small $h$ and appropriate coefficients it approximates $\Phi_H^h$ to order $r$.

\begin{exmp}\label{ex:strang}
    Let $H(x,p,s)=K(p) + V(x) + D(x,p,s)$ be a contact Hamiltonian, where $K$ is a kinetic energy term, $V$ is a potential energy term, and $D$ is a dissipation term. An $r=2$ splitting integrator with step $h$ for $H$ is given by the Strang splitting:
    \begin{equation}\label{eq:strang}
                \psi = \Phi_{D}^{\frac{h}{2}} \circ \Phi_{V}^{\frac{h}{2}}\circ \Phi_{K}^h \circ \Phi_{V}^{\frac{h}{2}} \circ \Phi_{D}^{\frac{h}{2}}.
    \end{equation}
    Higher-order integrators can be constructed by classical Yoshida--Suzuki composition methods \cite{yoshida1990construction,suzuki1990fractal}: the composition coefficients are chosen to cancel the leading-order error terms in the Baker--Campbell--Hausdorff expansion of $\psi$ around $\Phi_H^h$, and at higher integrator order some coefficients are necessarily negative, as in classical geometric composition methods.
\end{exmp}

\paragraph{Backward-error input.} The technical tool that transfers continuous statements to the discrete map is the contact analogue of shadow-Hamiltonian BEA: for an order-$r$ contact splitting on a compact set, sufficiently small $h$, and a chosen BEA truncation index $q$, there is a modified contact Hamiltonian $\tH_{h,q} = H + O(h^r)$ whose flow shadows the numerical iteration over the finite horizon up to a defect that is algebraic in $h$ at finite smoothness and exponentially small in $1/h$ under analyticity. The modified conformal factor, denoted $\tlam_{h,q}$, satisfies
\begin{equation}\label{eq:bea_conformal_local}
    \tlam_{h,q} := \partial_s \tH_{h,q} = \partial_s H + O(h^r),
\end{equation}
uniformly on the compact region. This input is packaged as \cref{ass:bea} below and verified for the splittings of this paper in \cref{app:bch}. The underlying contact Lie-algebra calculus is developed in~\cite{kevrekidis2026local}.

\subsection{Hamiltonians For Optimization}
\label{sec:master}
For optimization of a smooth objective $f:\R^n \to \R$, we will use as our main design template master contact Hamiltonians of the form
\begin{equation}\label{eq:master}
    \boxed{H(x,p,s) = K(x,p) + V(x) + D(x,p,s)}
\end{equation}
with potential $V(x)=f(x)$ determined by the objective. The choice of kinetic $K$ and dissipation $D$ terms then parametrizes the optimizer within this tractable subclass. This is a practical API rather than an exhaustive normal form. Interesting contact Hamiltonians used in optimization, such as the Bregman Hamiltonian, need not be naturally separable in this way and may require a contact change of variables or more general splitting machinery~\cite{bravetti2023bregman,kevrekidis2026local}. In the explicit optimizer families considered in this paper, the dissipative term is affine in the contact variable, $D(x,p,s)=d(x,p)s$, so the conformal rate is $\partial_s H=\partial_s D=d(x,p)$. The associated contact Hamiltonian equations have the form
\begin{equation}\label{eq:master_eom}
    \dot{x} = \nabla_p K + \nabla_p D\qc \dot{p} = -\nabla_x V - \nabla_x D - p \frac{\partial H}{\partial s}\qc \dot{s} = p^T(\nabla_p K + \nabla_p D) - H,
\end{equation}
with $\partial_s D$ determining the conformal factor of the flow. For Euclidean kinetic energy and scalar dissipation $D=D(s)$, projection to $\R^n$ gives the second-order ODE $\ddot x + (\partial_s D)\,\dot x + \nabla f(x) = 0$, recovering the familiar damped-system-with-potential-$f$ interpretation used in optimization dynamics \cite{polyak1964heavyball,su2016differential,wibisono2016variational}. If instead $D=\beta(x)s$, the projected momentum equation contains the additional contact force $-s\nabla\beta(x)$, as made explicit below. We define a contact algorithm as follows:

\begin{defn}[Contact Algorithm] A contact algorithm is a discrete-time optimization algorithm obtained by (a) determining a continuous-time contact Hamiltonian $H$ of the form~\eqref{eq:master} with potential $V=f$, and (b) discretizing the flow $\Phi_H^t$ using a contact splitting integrator $\Psi_h$ for some step size $h>0$. The resulting discrete-time algorithm is given by the iteration $z_{n+1} = \Psi_h z_n$, where $z_n = (x_n, p_n, s_n)$ is the state of the algorithm at iteration $n$.
\end{defn}

\subsection{Closed-form sub-flow catalogue}
\label{sec:subflows}

Splitting integrators are most attractive when the sub-flows can be solved in closed form. The master template~\eqref{eq:master} admits three structural strata, distinguished by which of the two exact-contact generator classes of~\cite{kevrekidis2026local} they belong to:
\begin{itemize}
    \item \textbf{Strict:} Hamiltonians independent of $s$. They generate contact flows that project to symplectic flows on $T^*\R^n$ and preserve $\alpha$ exactly (conformal factor $1$). Kinetic and potential terms live here.
    \item \textbf{Prolonged:} Hamiltonians affine in $p$, i.e.\ $a(x,s) + b^i(x,s)p_i$. Their flows are prolongations of dynamics on the base space $(x,s)$, meaning that the base motion determines the motion of $p$ uniquely through the chain rule that maintains $p = \nabla_x s$ along $1$-jets. Classical scalar dissipations live here.
    \item \textbf{Polynomial-in-$p$ with $s$-dependence:} Hamiltonians outside both classes. These are covered by the Lie-density theorem of~\cite{kevrekidis2026local} via commutator gadgets, but certain explicit families nonetheless admit closed-form sub-flows, as item (iv) below shows.
\end{itemize}
Each core atom below is an exact contactomorphism in closed form when its coefficients are fixed during the step. An $x$-dependent preconditioned kinetic Hamiltonian also generates an autonomous strict contact flow, but that geodesic flow is generally not available in closed form. The explicit preconditioned shear displayed below therefore requires the preconditioner to be constant during the sub-step. The implemented variant that re-evaluates the metric at the incoming point and applies this shear, Contact-Newton, is only approximately contact; see the caveat closing item (i). With that caveat, these atoms generate all algorithms used in this paper.

\paragraph{(i) Strict kinetic atoms.}
The Euclidean kinetic term $K(p) = \tfrac12\|p\|^2$ generates the shear
\begin{equation}\label{eq:K_eucl}
    \Phi_K^\tau \colon \quad
    x \mapsto x + \tau p,\quad
    p \mapsto p,\quad
    s \mapsto s + \tfrac\tau2\|p\|^2.
\end{equation}
The \emph{preconditioned kinetic} term $K(x,p) = \tfrac12 p^\top M(x)^{-1} p$ is independent of $s$, so its exact autonomous flow is strict contact and projects to the symplectic geodesic Hamiltonian flow
\begin{equation}\label{eq:K_metric_geodesic}
    \dot x = M(x)^{-1}p,\qquad
    \dot p_j = -\tfrac12 p^\top \partial_{x_j}\!\bigl(M(x)^{-1}\bigr)p,\qquad
    \dot s = \tfrac12 p^\top M(x)^{-1}p.
\end{equation}
This flow is generally implicit rather than available in closed form. Freezing a constant external metric $M_0$ during the sub-step reduces it to the explicit exact strict-contact shear
\begin{equation}\label{eq:K_metric}
    \Phi_K^\tau \colon \quad
    x \mapsto x + \tau\, M_0^{-1} p,\quad
    p \mapsto p,\quad
    s \mapsto s + \tfrac\tau2\,p^\top M_0^{-1} p.
\end{equation}
Instantiating $M(x)$ gives a family of preconditioned methods that spans the usual metric choices. Taking $M = I$ recovers~\eqref{eq:K_eucl}, $M(x) = \nabla^2 f(x)$ gives the second-order \emph{Contact-Newton}, $M$ from the L-BFGS two-loop recursion gives the quasi-Newton \emph{Contact-L-BFGS}, and $M = \mathrm{diag}\bigl(\sqrt{\hat v_t} + \varepsilon\bigr)$ with $\hat v_t$ an EMA of squared gradients is the Adam-style diagonal preconditioner.

How these rows are interpreted depends on how the preconditioner is updated. When it is an \emph{external state} updated between steps (the Adam EMA buffer $\hat v_t$, the L-BFGS pair history), $M_0$ is a constant parameter during the step, and the map above is an exact strict contactomorphism at every step, although a different one at each step. The resulting sequence of maps is nonautonomous, so the exact conformal bookkeeping of \cref{prop:exact_conformal} still applies but the single-modified-Hamiltonian backward error analysis of \cref{ass:bea} does not apply verbatim. When instead $M=M(x)$ is re-evaluated at the incoming point, as in the implemented Contact-Newton variant with $M=\nabla^2 f(x)$, applying the same frozen shear is \emph{not} an exact contact map. Its pullback acquires an $O(\tau\,\|\partial_x M\|)$ defect, computed in \cref{app:catalogue}. This is a defect of the incoming-point shear, not an obstruction to contact integration of~\eqref{eq:K_metric_geodesic}: the latter may instead be solved exactly when its geodesic flow is known, or approximated by an implicit symplectic method with its strict contact lift. Such alternatives preserve the autonomous geometric formulation but lie outside the closed-form splitting catalogue and the exact-sub-flow BCH verification used here.

\paragraph{(ii) Potential atom.}
The potential $V(x) = f(x)$ generates the momentum kick
\begin{equation}\label{eq:V_flow}
    \Phi_V^\tau \colon \quad
    x \mapsto x,\quad
    p \mapsto p - \tau\nabla f(x),\quad
    s \mapsto s - \tau f(x).
\end{equation}

\paragraph{(iii) Scalar damping atoms.}
Constant damping $D = \gamma s$ has conformal rate $\partial_s D = \gamma$ and generates the pure rescaling
\begin{equation}\label{eq:D_constant}
    \Phi_D^\tau \colon \quad
    x \mapsto x,\quad
    p \mapsto e^{-\gamma\tau} p,\quad
    s \mapsto e^{-\gamma\tau} s.
\end{equation}
State-dependent damping $D = \beta(x)\,s$ has a sub-flow that keeps $x$ frozen, so $\beta(x)$ and $\nabla\beta(x)$ are evaluated at the incoming position,
\begin{equation}\label{eq:D_statedep}
    \Phi_D^\tau \colon \quad
    x \mapsto x,\quad
    p \mapsto e^{-\beta(x)\tau} \bigl(p - \tau s\,\nabla\beta(x)\bigr),\quad
    s \mapsto e^{-\beta(x)\tau} s.
\end{equation}
State-dependent damping therefore contributes not only the conformal rescaling of $(p,s)$ but also the contact correction $-\tau s\,\nabla\beta(x)$ in the momentum. Nonlinear action damping $D = (\gamma/2)s^2$ has conformal rate $\partial_s D = \gamma s$, and its sub-flow is solved via the conserved quantity $p\,s^{-2}$,
\begin{equation}\label{eq:D_nonlinear}
    \Phi_D^\tau \colon \quad
    x \mapsto x,\quad
    s \mapsto \frac{s}{1 + \tfrac\gamma2\, s\, \tau},\quad
    p \mapsto p \Bigl( 1 + \tfrac\gamma2\,s\,\tau \Bigr)^{-2}.
\end{equation}
The denominator in~\eqref{eq:D_nonlinear} vanishes at $\tau = 2/(\gamma|s|)$ when $s<0$, which is the finite-time blow-up of the scalar Riccati equation $\dot s = -\tfrac\gamma2 s^2$ governing the action variable along this sub-flow, so implementations clamp $\tau$ strictly below that bound.

\paragraph{(iv) Momentum-dependent damping.}
The atom $D_{p} = \alpha\|p\|^2 s$ is neither strict (it depends on $s$) nor prolonged (it is quadratic in $p$), yet its sub-flow is exact in closed form: with $w_0 = \|p_0\|^2$ and $c = 1 + 2\alpha w_0\tau$,
\begin{equation}\label{eq:D_psq}
    \Phi_{D_p}^\tau \colon \quad
    x \mapsto x + 2\alpha\,\tau\, s\, p,\quad
    p \mapsto p/\sqrt{c},\quad
    s \mapsto s\sqrt{c},
\end{equation}
derived from two conserved quantities of the sub-flow in \cref{app:catalogue}. For $\alpha > 0$, $\|p\|$ decreases and $|s|$ increases, so momentum is \emph{transferred into the action variable}, damping large momenta more than small ones. The atom enters the framework either as a closed-form brick treated as a single atom in the outer composition, or as part of a polynomial-in-$p$ surrogate covered by the Lie-density machinery of~\cite{kevrekidis2026local}.

\paragraph{Composition.} Given $H = K + V + D$ with the sub-flows above, the second-order Strang composition is~\eqref{eq:strang}, and the fourth-order Yoshida triple jump is
\begin{equation}\label{eq:yoshida}
    \Psi^{(4)}_h \;=\;
    \Psi^{(2)}_{\gamma_1 h} \circ \Psi^{(2)}_{\gamma_0 h} \circ
    \Psi^{(2)}_{\gamma_1 h},
    \qquad
    \gamma_1 = \frac{1}{2 - 2^{1/3}},\ \gamma_0 = 1 - 2\gamma_1.
\end{equation}
Here $\gamma_0<0$, so the middle stage runs each atom backward. Individual damping factors may therefore exceed one, while their signed durations still combine to the intended net conformal exponent. Each factor is an exact contactomorphism, so every composition is a contactomorphism by construction. Structure preservation is not approximate.

\subsection{Algorithm design recipe}
\label{sec:recipe}

The pipeline that organizes the rest of the paper can be summarized as follows.

\medskip
\noindent\fbox{\parbox{\dimexpr\linewidth-2\fboxsep-2\fboxrule}{%
\textbf{Contact algorithm design recipe.}
\begin{enumerate}
    \item[(i)] Choose the master Hamiltonian $H = K + f + D$ from the atoms of \cref{sec:subflows}.
    \item[(ii)] Identify the conformal rate $\partial_s H$ and the candidate certificate built from $H$, in the shifted form introduced in \cref{sec:main}.
    \item[(iii)] Discretize with a contact splitting built from the closed-form sub-flows.
    \item[(iv)] Verify (or assume) the certificate hypotheses stated in \cref{sec:main}, in particular the comparison bounding the objective gap by the certificate on a compact region, and invoke \cref{thm:rate} to transfer the continuous rate certificate to the discrete iterates over the backward-error horizon.
\end{enumerate}}}
\medskip

The geometry supplies the decay identity for free, but only the verified comparison of step (iv) turns Hamiltonian decay into objective decay. \Cref{sec:closed_form} carries out that verification in closed form for representative families.

\subsection{Recovery of classical algorithms}
\label{sec:recovery}

\begin{table}[t]
\centering
\renewcommand{\arraystretch}{1.4}
\begin{tabular}{l l l l}
\hline
$K(x,p)$ & $D(x,p,s)$ & $\partial_s H$ & Algorithm \\
\hline
$\tfrac12\|p\|^2$ & $\gamma s$ & $\gamma$ & Heavy ball~\cite{polyak1964heavyball} \\
$\tfrac12\|p\|^2$ & $2\sqrt\mu\,s$ & $2\sqrt\mu$ & NAG-type strongly convex ODE~\cite{nesterov1983method,wibisono2016variational} \\
$\tfrac12 p^\top\!\nabla^2\!f(x)^{-1}p$ & $\gamma s$ & $\gamma$ & Contact-Newton \\
$\tfrac12 p^\top\! B_\text{LBFGS}^{-1} p$ & $\gamma s$ & $\gamma$ & Contact-L-BFGS \\
$\tfrac12 p^\top\!\mathrm{diag}(\sqrt{\hat v}+\varepsilon)^{-1}p$
    & $\gamma s$ & $\gamma$ & Contact-Adam (C-Adam) \\
$\tfrac12\|p\|^2$ & $\beta(x)s$ & $\beta(x)$ & State-adaptive damping [new] \\
$\tfrac12\|p\|^2$ & $(\gamma/2)s^2$ & $\gamma s$ & Self-regulating damping [new] \\
$\tfrac12\|p\|^2$ & $\alpha\|p\|^2 s$ & $\alpha\|p\|^2$ & Momentum-transfer damping [new] \\
\hline
\end{tabular}
\caption{\textbf{Autonomous} contact-Hamiltonian templates. The potential is $V(x)=f(x)$ in all rows, and the conformal rate is determined by $\partial_s H$. Rows whose sub-flows are exact (or external-state frozen, see \cref{sec:subflows}) fall directly under \cref{thm:rate} whenever the certificate hypotheses of \cref{ass:cert} are verified. The variable-metric Contact-Newton Hamiltonian has an autonomous strict contact flow, but the implemented incoming-point frozen shear is only approximately contact and is therefore included as a framework-motivated variant.}
\label{tab:classical_algos}
\end{table}

\begin{table}[t]
\centering
\renewcommand{\arraystretch}{1.4}
\begin{tabular}{l l l l}
\hline
$K(x,p)$ & $D(x,p,s,t)$ & $\partial_s H$ & Algorithm \\
\hline
$\tfrac12\|p\|^2$ & $\tfrac{r}{t}s$ & $\tfrac{r}{t}$ & Nesterov ODE, $r\ge 3$~\cite{su2016differential} \\
$\tfrac12\|p\|^2$ & $\beta_k s$ (iteration-scheduled) & $\beta_k$ & C-NAG / scheduled contact-SGD (\cref{app:algorithms}) \\
\hline
\end{tabular}
\caption{\textbf{Time-dependent} damping models. These are nonautonomous and require the autonomous lift of \cref{prop:lift} to enter the geometric framework of \cref{sec:main}. Application of the rate-transfer theorem additionally requires verification of the lifted certificate hypotheses. We include them as dictionary entries connecting the contact viewpoint to classical acceleration.}
\label{tab:classical_algos_td}
\end{table}

\paragraph{Relation to Existing Methods.} \Cref{tab:classical_algos} distinguishes \emph{autonomous contact templates}, which fall directly under the theory of \cref{sec:main}, from the \emph{time-dependent accelerated models} of \cref{tab:classical_algos_td}, which require an autonomous lift before the same geometric analysis applies. Familiar continuous-time optimization models are recovered by fixing the kinetic term and choosing the contact dissipation $D$ so that $\partial_s H$ matches the desired damping law. The preconditioned kinetic rows remain strict, in the sense that $\partial_s K=0$, so the conformal rate is still set entirely by $D$. Preconditioning changes how efficiently kinetic energy is converted into progress on $f$, not the prescribed contact-form multiplier. The last three rows of \cref{tab:classical_algos} illustrate contact-specific possibilities: trajectory-dependent damping $\partial_s H=\beta(x)$, where the accumulated conformal factor depends on the region of the landscape visited by the flow; self-regulating damping $\partial_s H = \gamma s$, where the rate is set by the accumulated action; and momentum-transfer damping $\partial_s H = \alpha\|p\|^2$, which damps large momenta more than small ones. Their state-dependent multipliers are represented intrinsically on the same odd-dimensional state space by the contact formalism. Enlarged symplectic realizations may also be constructed, but they do not provide this same-space conformal description.

For the time-dependent rows, the lift is recorded as \cref{prop:lift} in \cref{app:lift}: augmenting the base with a time coordinate $\theta$ and its conjugate momentum $\pi$, the lifted Hamiltonian $\bar H = \tfrac12\|p\|^2 + f(x) + \pi + \beta(\theta)\,s$ on $J^1(\R^{n+1})$ is again of master form (two strict atoms plus one prolonged atom), reproduces the nonautonomous dynamics $\ddot x + \beta(t)\dot x + \nabla f(x)=0$, and has conformal rate $\partial_s\bar H = \beta(\theta)$. Nonautonomous damping laws therefore \emph{are} contact Hamiltonian, but on an enlarged phase space, and the free momentum $\pi$ means that positivity and objective comparison do not follow from the lift alone. For the Nesterov choice $\beta(t)=r/t$, all compactness hypotheses of \cref{sec:main} must be imposed away from the $t=0$ singularity, on windows $[t_0,t_0+T]$ with $t_0>0$. In the remainder of the paper, all certified statements refer to the autonomous setting unless the lift is invoked explicitly and its certificate hypotheses are checked.

\section{Certificate-Transfer Theory}\label{sec:main}
\subsection{From Hamiltonian decay to optimization convergence}
\label{sec:rates}

The identity~\eqref{eq:Hdot} states that the Hamiltonian decays geometrically at the pointwise rate $\partial_s H$. For the master Hamiltonian $H = K + V + D$, this gives an augmented certificate for optimization once the non-potential terms are controlled along the trajectory. Fix a constant reference value $H^\star$ (in the certified examples below, $H^\star = f^\star := \min_x f$) and set
\begin{equation}\label{eq:mcE_def}
    \mcE(z) \;:=\; H(z) \,-\, H^\star.
\end{equation}
The Hamiltonian certificate becomes an optimization certificate once the objective gap is controlled by the augmented energy. We will use the comparison condition
\begin{equation}
    \label{eq:objective_comparison}
    V(x)-f^\star \le C_{\mathrm{cmp}}\,\mcE(x,p,s)
\end{equation}
along the trajectories under consideration. A simple sufficient case is $K(x,p)+D(x,p,s)\ge 0$ and $H^\star=f^\star$, for which
\begin{equation*}
    V(x)-f^\star
    \le H(x,p,s)-f^\star
    = \mcE(x,p,s)
\end{equation*}
holds with $C_{\mathrm{cmp}}=1$. This pointwise positivity is convenient but stronger than necessary. The closed-form quadratic analysis of \cref{sec:quadratic} shows that $K+D$ can dip negative transiently while the comparison~\eqref{eq:objective_comparison} still holds with an explicit constant $C_{\mathrm{cmp}}>1$. More generally, an objective-level certificate may be supplied by a problem-specific Lyapunov construction that is not of the Hamiltonian form (\cref{lem:bregman_lyap}, \cref{app:lyapunov}). Such certificates transfer to the iterates by the shadowing route (\cref{cor:aux_transfer}) rather than through the conformal identity. The theorem below is stated conditionally on the Hamiltonian comparison~\eqref{eq:objective_comparison}. For the shifted Hamiltonian certificate, $H^\star\ge0$ and $\partial_sH\ge0$ imply directly that
\begin{equation}\label{eq:E_decay_assumption}
    \frac{\dx}{\dx t}\,\mcE\bigl(\Phi^t_H z_0\bigr)
    \le -\partial_s H\bigl(\Phi^t_H z_0\bigr)\,
    \mcE\bigl(\Phi^t_H z_0\bigr)
\end{equation}
on the trajectory, since $\dot\mcE=-\partial_sH\,(\mcE+H^\star)$. Gr\"onwall's inequality then gives
\begin{equation}\label{eq:E_decay}
    \mcE\bigl(\Phi^t_H z_0\bigr)
    \;\le\; \Rate_H(t)\,\mcE(z_0),
    \qquad
    \Rate_H(t) \;:=\; \exp\!\Big(\!-\!\int_0^t \partial_s H\bigl(\Phi^\tau_H z_0\bigr)\,d\tau\Big).
\end{equation}
Combining~\eqref{eq:E_decay} with~\eqref{eq:objective_comparison} gives the corresponding objective envelope $V(x_t)-f^\star \le C_{\mathrm{cmp}}\Rate_H(t)\mcE(z_0)$. Four canonical specializations make the dependence on the conformal factor explicit:
\begin{itemize}
    \item[(a)] \textbf{Heavy ball}, $D = \gamma s$:\; $\partial_s H = \gamma$ is constant, so $\Rate_H(t) = e^{-\gamma t}$.
    \item[(b)] \textbf{Strongly-convex regime}, $D = 2\sqrt\mu\,s$:\; $\partial_s H = 2\sqrt\mu$, so $\Rate_H(t) = e^{-2\sqrt\mu\,t}$.
    \item[(c)] \textbf{State-dependent damping}, $D = \beta(x)s$:\; $\partial_s H = \beta(x)$ is variable, and $\Rate_H(t) = \exp\!\big(-\!\int_0^t \beta(x(\tau))\,d\tau\big)$ tracks the trajectory-weighted accumulated damping. If $\beta(x(t))\ge 0$ along the trajectory, the envelope is monotone.
    \item[(d)] \textbf{Nesterov polynomial regime} (lifted, \cref{prop:lift}), $\beta(t) = r/t$:\; $\Rate_H(t) = (t_0/(t_0+t))^{r}$, the polynomial Hamiltonian envelope. The objective-level rate certified by the Su--Boyd--Cand\`es functional is $O(1/t^2)$ (\cref{app:lyapunov}), and the two statements are distinct and should not be conflated.
\end{itemize}
In each case the rate is the integral of the conformal factor along the flow. We will write the rate certificate of $H$ as the pair $(\mcE,\partial_s H)$, implicitly understood to mean the decay~\eqref{eq:E_decay} of $\mcE$ along $\Phi^t_H$.

The discussion above isolates the scope of the discrete theorem. The contact identity gives a rate certificate for the Hamiltonian itself, and objective-level rates follow only on regions where the comparison, nonnegativity, and regularity assumptions hold along the shadowed trajectory. We now state these assumptions once, as named hypotheses, and quote them in the theorem.

\subsection{Assumptions}
\label{sec:assumptions}

\begin{assumption}[Regularity and compactness]\label{ass:reg}
$U \subset J^1(\R^n)$ is compact; the Hamiltonian $H$ and the sub-flows defining the order-$r$ contact splitting $\Psi^{(r)}_h$ are of class $C^{q+2}$ on a neighborhood of $U$ for a BEA truncation order $q \ge r$; a finite horizon $T>0$ is fixed; and the continuous, modified, and numerical trajectories issued from the initial conditions under consideration remain in $U$ for $nh\le T$ and sufficiently small $h$. For state-dependent damping $D=\beta(x)s$ this in particular requires $\beta \in C^{q+2}$ on a neighborhood of $U$, and the Lipschitz constants of $\beta$ and $\nabla\beta$ on $U$ enter the estimates below. Optionally, $H$ and the sub-flows are real analytic with a holomorphic extension to a complex neighborhood of $U$.
\end{assumption}

\begin{assumption}[Positive certificate region]\label{ass:cert}
On $U$ and along the trajectories under consideration:
\begin{enumerate}
    \item[(i)] the reference value is the objective minimum, $H^\star = f^\star \ge 0$, and the certificate is nonnegative, $\mcE = H - H^\star \ge 0$ (nonnegativity of $f^\star$ is a normalization, available by replacing $f$ with $f - f^\star$; see \cref{rem:gauge} for the one case where this shift is not innocuous);
    \item[(ii)] the damping is nonnegative, $\partial_s H \ge 0$ on $U$ (so that $\Rate_H(t) \le 1$);
    \item[(iii)] the comparison condition~\eqref{eq:objective_comparison} holds with a finite constant $C_{\mathrm{cmp}}$;
\end{enumerate}
\end{assumption}

\begin{assumption}[Backward-error input]\label{ass:bea}
For sufficiently small $h$ there exists a modified contact Hamiltonian $\tH_{h,q}$ on a compact $U' \subseteq U$ containing the numerical trajectory such that:
\begin{enumerate}
    \item[(i)] $\tH_{h,q} = H + O(h^r)$ and $\partial_s \tH_{h,q} = \partial_s H + O(h^r)$, uniformly in $C^1(U')$, and hence $\mathrm{Lip}_{U'}(\tH_{h,q} - H) = O(h^r)$;
    \item[(ii)] the modified flow shadows the splitting, $\bigl\|(\Psi^{(r)}_h)^n - \Phi^{nh}_{\tH_{h,q}}\bigr\|_{C^0(U')} \le \delta_{\mathrm{BEA}}(h,T)$ for $nh\le T$, where the state-space defect satisfies $\delta_{\mathrm{BEA}}(h,T) = O_T(h^q)$ at finite smoothness and $\delta_{\mathrm{BEA}}(h,T)=O(e^{-c/h})$ in the analytic case;
    \item[(iii)] the certificate-level defect is defined from the state-space defect through a Lipschitz constant of the \emph{modified} certificate, $\rho_{\mathrm{BEA}}(h,T) := L_{\tE}\,\delta_{\mathrm{BEA}}(h,T)$ with $L_{\tE} := \mathrm{Lip}_{U'}(\tE_{h,q})$; by item \textup{(i)}, $L_{\tE} \le \mathrm{Lip}_{U'}(\mcE) + O(h^r)$, which is finite and $h$-uniformly bounded by \cref{ass:reg}.
\end{enumerate}
\Cref{app:bch} verifies this assumption for master splittings built from the closed-form sub-flows of \cref{sec:subflows} under \cref{ass:reg}, via the contact Baker--Campbell--Hausdorff expansion; see also~\cite{kevrekidis2026local,hairer2006geometric}.
\end{assumption}

Three comments on the hypotheses. First, \cref{ass:cert} carries the optimization-specific content, since the geometry does not supply nonnegativity or the objective comparison for free, and \cref{sec:closed_form} verifies them in closed form for concrete families. Second, the defect $\rho_{\mathrm{BEA}}$ in the theorem is a \emph{certificate} defect. \Cref{ass:bea}(iii) makes explicit the Lipschitz conversion from the state-space shadowing distance, which is often left implicit in BEA statements, and item (ii) of \cref{ass:cert}, nonnegative damping, licenses the additive form of the envelope perturbation in~\eqref{eq:envelope} below, with the sign-indefinite case discussed in \cref{rem:sign_indef}. Third, two useful consequences need not be imposed separately. Items (i)--(ii) of \cref{ass:cert} imply the continuous decay~\eqref{eq:E_decay_assumption}, while \cref{ass:bea}(i) implies $\sup_{U'}|\tE_{h,q}-\mcE|=O(h^r)$. The corresponding decay of the \emph{modified} certificate is likewise derived, rather than assumed, because the multiplicative contact identity~\eqref{eq:Hdot} holds exactly for $\tH_{h,q}$ along its own flow. We record this as a lemma.

\begin{lemma}[Inherited decay of the modified certificate]\label{lem:mod_decay}
Assume \cref{ass:reg}, items \textup{(i)}--\textup{(ii)} of \cref{ass:cert}, and item \textup{(i)} of \cref{ass:bea}. Then for sufficiently small $h$, the modified certificate $\tE_{h,q} = \tH_{h,q} - H^\star$ satisfies, along the modified flow in $U'$,
\begin{equation}\label{eq:inherited_decay}
    \frac{\dx}{\dx t}\,\tE_{h,q} \;\le\; -\,\partial_s\tH_{h,q}\;\tE_{h,q} \;+\; C\,h^r H^\star,
\end{equation}
for a constant $C$ independent of $h$, and consequently $\tE_{h,q}\bigl(\Phi^t_{\tH_{h,q}}z_0\bigr) \le \widetilde{\Rate}_{h,q}(t)\,\tE_{h,q}(z_0) + C_0\,h^r$ for all $t\le T$, with $\widetilde{\Rate}_{h,q}$ as in~\eqref{eq:mod_decay} and with $C_0$ independent of $h$; in the normalized case $H^\star = 0$, one has $C_0 = 0$.
\end{lemma}

\begin{proof}
The decay identity~\eqref{eq:Hdot} holds \emph{exactly} for the modified Hamiltonian along its own flow: $\frac{\dx}{\dx t}\tH_{h,q} = -\tH_{h,q}\,\partial_s\tH_{h,q}$. Writing $\tH_{h,q} = \tE_{h,q} + H^\star$,
\[
    \frac{\dx}{\dx t}\,\tE_{h,q} \;=\; -\,\partial_s\tH_{h,q}\,\tE_{h,q} \;-\; \partial_s\tH_{h,q}\,H^\star .
\]
By \cref{ass:cert}\textup{(ii)} and \cref{ass:bea}\textup{(i)}, $\partial_s\tH_{h,q} \ge -Ch^r$ on $U'$, and $H^\star \ge 0$ by \cref{ass:cert}\textup{(i)}, so $-\partial_s\tH_{h,q}\,H^\star \le C h^r H^\star$, which is~\eqref{eq:inherited_decay}. Gr\"onwall over $[0,T]$ gives the integrated form, the accumulated factor $e^{Ch^r T}$ being bounded for small $h$. When $H^\star = 0$ the forcing term is absent and the decay is exact.
\end{proof}

The normalization $f^\star \ge 0$ in \cref{ass:cert}(i) is innocuous for constant damping, but for state-dependent damping a constant shift of $f$ feeds back into the dynamics through the contact correction $s\,\nabla\beta$, so a lower reference for $f$ is a genuine hypothesis there; see \cref{rem:gauge} in \cref{app:proofs}.

\subsection{The rate-transfer theorem}

The theorem below is a \emph{certificate-transfer result}, not a global convergence theorem for arbitrary contact optimizers. It shows that once an autonomous contact Hamiltonian admits a continuous-time certificate $(\mcE,\partial_s H)$ on a compact region, an order-$r$ contact splitting transfers the associated finite-horizon decay envelope up to the usual modified-equation errors.

\begin{theorem}[Rate transfer via modified conformal factor]\label{thm:rate}
Let the Regularity and Compactness hypothesis (\cref{ass:reg}), the Positive Certificate Region hypothesis (\cref{ass:cert}), and the Backward-Error Input (\cref{ass:bea}) hold for the contact Hamiltonian $H$, the order-$r$ contact splitting $\Psi^{(r)}_h$, the compact region $U$, the horizon $T$, and the truncation order $q\ge r$. Then, for sufficiently small $h$ and $z_n=(\Psi^{(r)}_h)^n z_0$, the following hold for all $nh \le T$:
\begin{enumerate}
    \item[\textup{(i)}] \textbf{Modified certificate decay} (implied by \cref{lem:mod_decay}). Along the modified flow,
    \begin{equation}\label{eq:mod_decay}
        \tE_{h,q}\bigl(\Phi^{t}_{\tH_{h,q}} z_0\bigr) \;\le\; \widetilde{\Rate}_{h,q}(t)\;\tE_{h,q}(z_0) \;+\; C_0\,h^r,
        \qquad
        \widetilde{\Rate}_{h,q}(t) := \exp\!\Bigl(-\int_0^t \partial_s \tH_{h,q}\circ\Phi_{\tH_{h,q}}^\tau\,d\tau\Bigr),
    \end{equation}
    with $C_0 = 0$ when $H^\star = 0$.
    \item[\textup{(ii)}] \textbf{Conformal-factor perturbation.} The modified envelope tracks the continuous one,
    \begin{equation}\label{eq:env_perturb}
        \widetilde{\Rate}_{h,q}(t) \;\le\; \Rate_H(t) \;+\; C_T\,h^r\,t,
        \qquad
        \Rate_H(t):=\exp\!\Bigl(-\int_0^t \partial_s H\circ\Phi_H^\tau\,d\tau\Bigr),
    \end{equation}
    where $C_T$ depends on $H$, $U$, $T$ but not on $h$. The bound uses $\Rate_H\le 1$ from \cref{ass:cert}\textup{(ii)}.
    \item[\textup{(iii)}] \textbf{Discrete transfer.} Consequently,
    \begin{equation}\label{eq:envelope}
        \mcE(z_n)
        \;\le\;
        \Bigl[\Rate_H(nh) + C_T\,h^r\,nh\Bigr]\mcE(z_0)
        \;+\;
        \rho_{\mathrm{BEA}}(h,T) \;+\; C_{\mathrm{mod},T}\,h^r,
    \end{equation}
    where $\rho_{\mathrm{BEA}}$ is the certificate defect of \cref{ass:bea}\textup{(iii)} and $C_{\mathrm{mod},T}\,h^r$ collects the modified-certificate comparison implied by \cref{ass:bea}\textup{(i)} and the forcing term in \cref{lem:mod_decay}. By the comparison hypothesis \cref{ass:cert}\textup{(iii)},
    \begin{equation}\label{eq:objective_envelope}
        V(x_n)-f^\star
        \;\le\;
        C_{\mathrm{cmp}}\Bigl(\Bigl[\Rate_H(nh) + C_T\,h^r\,nh\Bigr]\mcE(z_0)
        \;+\;
        \rho_{\mathrm{BEA}}(h,T) + C_{\mathrm{mod},T}\,h^r\Bigr).
    \end{equation}
\end{enumerate}
\end{theorem}

\begin{proof}
See \cref{app:proofs}. In outline: \textup{(i)} is \cref{lem:mod_decay}; \textup{(ii)} follows from a Gr\"onwall comparison of the modified and original flows, using $\Rate_H \le 1$ from \cref{ass:cert}\textup{(ii)}; \textup{(iii)} splits the certificate error into a comparison term, a shadowing term, and the modified decay from \textup{(i)}--\textup{(ii)}. The verification of \cref{ass:bea} for the master splittings used in this paper is carried out in \cref{app:bch}.
\end{proof}

The modified Hamiltonian $\tH_{h,q}$ is the contact analogue of the classical shadow Hamiltonian of symplectic integrators, and can be computed explicitly as a truncated series in $h$ using the contact Baker--Campbell--Hausdorff formula (\cref{app:bch}). Analyticity is not required for the rate-transfer mechanism itself. It is the assumption that upgrades the algebraic BEA defect to one that is exponentially small in $1/h$. The contact analogue of the symplectic principle ``shadow Hamiltonian $\Rightarrow$ modified energy $\Rightarrow$ long-time stability'' is:
\begin{quotation}
``shadow Hamiltonian $\Rightarrow$ modified conformal factor $\Rightarrow$ rate-envelope preservation.''    
\end{quotation}

\begin{remark}[Beyond the backward-error horizon]\label{rem:bea_horizon}
The restriction $nh\le T$ is a limit of the certificate supplied by the present backward-error argument and \textit{not} an instability threshold for the numerical method. Once the controlled shadowing horizon is exceeded, \cref{thm:rate} no longer guarantees that the original continuous-time envelope bounds the discrete certificate with the stated defect. It does \emph{not} follow that the iterates diverge, or even that their convergence rate deteriorates: the method may remain stable and converge for arbitrarily longer times. Establishing such behavior requires a separate argument, for example a discrete Lyapunov estimate, contractivity or spectral analysis, or a renewed local shadowing argument. The exact contactness of a composition of exact contact sub-flows likewise persists for every iterate for which the maps are defined; what expires at $T$ is the quantitative comparison with the chosen modified flow and certificate.
\end{remark}

\begin{corollary}[State-dependent damping]\label{cor:state_dep}
For the contact family $H(x,p,s)=\tfrac12\|p\|^2+f(x)+\beta(x)s$, the conformal rate is $\partial_sH=\beta(x)$. If, in addition to \cref{ass:reg,ass:cert,ass:bea}, $\beta(x)\ge \beta_{\min}>0$ on $U$, then $\Rate_H(t)\le e^{-\beta_{\min}t}$ and
\begin{equation}\label{eq:state_dep_envelope}
    \mcE(z_n)
    \;\le\;
    \Bigl[e^{-\beta_{\min}nh}+C_T\,h^r\,nh\Bigr]\mcE(z_0)
    \;+\;
    \rho_{\mathrm{BEA}}(h,T) + C_{\mathrm{mod},T}h^r,
\end{equation}
with the objective gap bounded by the same expression multiplied by $C_{\mathrm{cmp}}$. Thus $\beta(x)$ acts as a design parameter. The damping can react to the region of the landscape visited by the optimizer while retaining the same backward-error rate-transfer mechanism whenever the Hamiltonian certificate hypotheses are verified. Strongly convex objectives also admit the separate Bregman-type Lyapunov route of \cref{lem:bregman_lyap}. That certificate is not of Hamiltonian form and therefore transfers through \cref{cor:aux_transfer}, not through this corollary.
\end{corollary}

\subsection{Transfer of auxiliary Lyapunov certificates}
\label{sec:aux_transfer}

Objective-level certificates need not be of the Hamiltonian form $\mcE = H - H^\star$. A problem-specific Lyapunov construction, such as the Bregman-type Lyapunov certificate for state-dependent damping established in \cref{lem:bregman_lyap}, controls the objective gap without being tied to the conformal identity. Such certificates still transfer to the discrete iterates, through trajectory shadowing, at the price of an additive defect.

\begin{corollary}[Discrete transfer of auxiliary Lyapunov certificates]\label{cor:aux_transfer}
Let \cref{ass:reg} and \cref{ass:bea} hold, and let $W$ be $C^1$ on a neighborhood of $U$ with $W \ge 0$, $\frac{\dx}{\dx t} W(\Phi^t_H z_0) \le -\lambda_W\, W(\Phi^t_H z_0)$ for some $\lambda_W > 0$, and $f - f^\star \le C_W\, W$ on $U$. Then, for sufficiently small $h$ and $nh \le T$,
\begin{equation}\label{eq:aux_transfer}
    W(z_n) \;\le\; e^{-\lambda_W\, nh}\, W(z_0) \;+\; L_W\bigl(C_T'\,h^r + \delta_{\mathrm{BEA}}(h,T)\bigr),
    \qquad
    f(x_n) - f^\star \;\le\; C_W\, W(z_n),
\end{equation}
with $L_W := \mathrm{Lip}_{U'}(W)$ finite by compactness.
\end{corollary}

\begin{proof}
See \cref{app:proofs}.
\end{proof}

The two routes are complementary. The conformal route of \cref{thm:rate} upgrades the mechanism from ``the trajectories are close'' to ``the dissipation bookkeeping is exact'' (\cref{prop:exact_conformal}), but it is tied to the Hamiltonian certificate. The shadowing route~\eqref{eq:aux_transfer} is blind to the conformal structure but accepts arbitrary certificates, in particular the functional of \cref{lem:bregman_lyap}.

\section{Closed-Form Verified Instances}
\label{sec:closed_form}

This section carries out step (iv) of the design recipe, certificate verification, in closed form for two representative families, and records the structural exactness property that sharpens the $O(h^r)$ accounting of \cref{thm:rate}. All splittings below use the closed-form atoms of \cref{sec:subflows}, and the closed forms have been verified against direct numerical integration at machine precision.

\subsection{Quadratic theory: a contact-certificate benchmark}
\label{sec:quadratic}

The first family is the constant-dissipation heavy-ball regime, where the certificate, the comparison constant, and the discrete map are all available in closed form, and where the certificate \emph{is} of Hamiltonian form. Dissipative leapfrog and its quadratic spectrum have already been analyzed through conformal-symplectic and presymplectic backward error analysis~\cite{francca2020conformal,francca2021dissipative}. Here the quadratic problem serves a different role, as a fully solvable benchmark showing how the augmented contact Hamiltonian becomes an objective-controlling certificate and how the hypotheses of \cref{thm:rate} can be checked sharply. Let $f(x) = \tfrac12 x^TAx$ with $A=A^T\succ0$, and write $\mu:=\lambda_{\min}(A)$ and $L:=\lambda_{\max}(A)$. Set
\begin{equation}
    H(x,p,s) = \frac{1}{2}\|p\|^2 + \frac{1}{2} x^T A x + \gamma s,
    \qquad 0<\gamma<2\sqrt{\mu}.
\end{equation}
The associated contact Hamiltonian equations are
\begin{equation}
    \dot{x} = p\qc \dot{p} = -A x - \gamma p\qc \dot{s} = \frac{1}{2}\|p\|^2 - \frac{1}{2} x^T A x - \gamma s,
\end{equation}
or, after eliminating the momentum, $\ddot{x} + \gamma \dot{x} + A x = 0$. Diagonalizing $A = Q\Lambda Q^T$ reduces the flow to damped oscillators $\ddot{y}_i + \gamma \dot{y}_i + \lambda_i y_i = 0$ in each eigendirection, with explicit underdamped solutions for $\gamma < 2\sqrt{\lambda_i}$.

\begin{prop}[Quadratic heavy ball: contact-certificate comparison and projected spectrum]\label{prop:quadratic}
Let $f(x) = \tfrac12 x^T A x$ with $A=A^T\succ0$ and $\mu=\lambda_{\min}(A)$, and let $H = \tfrac12\|p\|^2 + f(x) + \gamma s$ with $0 < \gamma < 2\sqrt\mu$, initialized at $p_0 = 0$, $s_0 = 0$. Then:
\begin{enumerate}
    \item[(i)] \textup{(Continuous certificate.)} $H^\star = f^\star = 0$, the certificate decays exactly, $\mcE(t) = e^{-\gamma t}\,\mcE(0) \ge 0$, and the comparison~\eqref{eq:objective_comparison} holds along the trajectory with the sharp constant
    \[
        C_{\mathrm{cmp}} \;=\; \frac{4\mu}{4\mu - \gamma^2}.
    \]
    This sharp comparison is along the continuous trajectory. The restriction $\gamma< 2\sqrt\mu$ is necessary for a horizon-uniform constant, since for any overdamped mode the ratio $(V-f^\star)/\mcE$ grows without bound.
    \item[(ii)] \textup{(Projected dissipative-leapfrog spectrum.)} For the Strang splitting~\eqref{eq:strang}, the $(x,p)$ dynamics decouple mode-wise into the explicit linear maps
    \[
        M_i(h) = \begin{pmatrix} 1 - h v_i & \delta h \\ -\delta v_i(2-hv_i) & \delta^2(1-hv_i)\end{pmatrix},
        \qquad v_i = \tfrac{\lambda_i h}{2},\quad \delta = e^{-\gamma h/2},
    \]
    with $\det M_i(h) = e^{-\gamma h}$ \emph{exactly}. Whenever $|1 - \lambda_i h^2/2| < \operatorname{sech}(\gamma h/2)$ (the oscillatory window), the eigenvalues of $M_i$ are complex with modulus \emph{exactly} $e^{-\gamma h/2}$, so the discrete spectral contraction rate equals the continuous rate with zero error, while the numerical angle satisfies $\theta_i(h) = \omega_i h + O(h^3)$, where $\omega_i := \sqrt{\lambda_i-\gamma^2/4}$ is the continuous underdamped frequency. This spectral statement does not make the full state error phase-only, because the numerical eigenvectors and the sine--cosine mixing of physical initial data also carry discretization error.
\end{enumerate}
\end{prop}

\begin{proof}
See \cref{app:quadratic_details}.
\end{proof}

The proposition rests on three closed-form facts. First, the contact identity gives $H(t) = e^{-\gamma t}H(0)$ with no inequality, so with $p_0=0$, $s_0=0$ the certificate is exact, $\mcE(t)=e^{-\gamma t}f(x_0)$, and the action variable requires no separate integration:
\begin{equation}\label{eq:s_closed_form}
    s(t) \;=\; \frac{1}{\gamma}\Bigl(e^{-\gamma t}H(0) - \tfrac12\|p(t)\|^2 - f(x(t))\Bigr).
\end{equation}
Second, along each mode $f$ oscillates inside the envelope, $\tfrac12\lambda_i y_i(t)^2 = e^{-\gamma t}\,\tfrac12\lambda_i y_{0,i}^2\, g_i(t)^2$ with $g_i(t) = \cos\omega_i t + \tfrac{\gamma}{2\omega_i}\sin\omega_i t$ and $\sup_t g_i^2 = 4\lambda_i/(4\lambda_i-\gamma^2)$ (\cref{app:quadratic_details}), so $K + D = \mcE - (V-f^\star)$ dips \emph{negative} whenever $g_i^2>1$. The convenient pointwise condition $K+D\ge0$ therefore fails transiently even in this simplest example, and the comparison-with-constant formulation of \cref{ass:cert}(iii), not pointwise positivity, is the right hypothesis. Moreover, as $\gamma \uparrow 2\sqrt\mu$ the certified envelope rate increases toward the optimal $2\sqrt\mu$ while $C_{\mathrm{cmp}} \to \infty$, so the classical critical-damping tradeoff appears here as a certificate-constant blowup. Third, the projected Strang map is an adjoint ordering of the dissipative leapfrog analyzed in~\cite{francca2020conformal,francca2021dissipative}. It has the same characteristic polynomial and therefore reproduces the known exact conformal spectral contraction. In \cref{fig:quad_envelope}, the Hamiltonian along the iterates stays close to the continuous envelope even at coarse steps (left), while the deviation of the discrete spectral contraction factor from $e^{-\gamma h/2}$ sits at the $10^{-16}$ floating-point floor throughout the oscillatory window (right). The explicit window and modal-shape calculation are recorded in \cref{app:quadratic_details}. At $\gamma=0$ the window reduces to the St\"ormer--Verlet stability interval $\lambda h^2 < 4$~\cite{hairer2006geometric}.

\begin{figure}[t]
    \centering
    \includegraphics[width=0.98\textwidth]{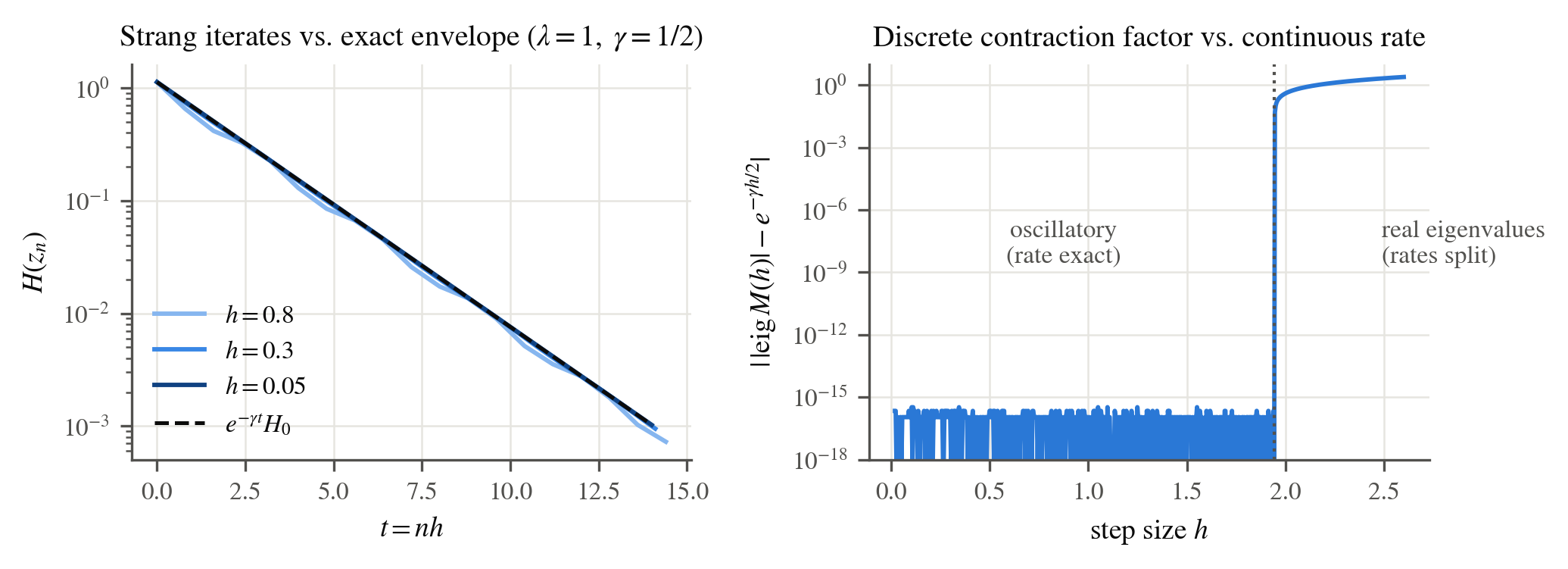}
    \caption{Quadratic contact-certificate benchmark (\cref{prop:quadratic}). \textbf{Left:} the augmented contact Hamiltonian along Strang iterates for three step sizes vs.\ the exact envelope $e^{-\gamma t}H_0$ ($\lambda=1$, $\gamma=1/2$), and even $h=0.8$ tracks the envelope with a bounded ripple. \textbf{Right:} reproduction of the exact spectral contraction known for dissipative leapfrog~\cite{francca2020conformal,francca2021dissipative}: deviation of $|\mathrm{eig}\,M(h)|$ from $e^{-\gamma h/2}$. Inside the oscillatory window the deviation is at floating-point level, although the phase and modal shape are not exact. Past $|1-\lambda h^2/2| = \operatorname{sech}(\gamma h/2)$ (dotted), the eigenvalues turn real and the rates split.}
    \label{fig:quad_envelope}
\end{figure}

\subsection{Exact discrete conformal bookkeeping}
\label{sec:exact_bookkeeping}

The known exact modal contraction of dissipative leapfrog is consistent with a structural property of the particular master splittings used here, namely that their contact-form multiplier can be evaluated exactly. We record that property in contact variables because it sharpens the $O(h^r)$ accounting of \cref{thm:rate} and extends directly to the state-dependent damping atoms below.

\begin{prop}[Exact discrete conformal bookkeeping]\label{prop:exact_conformal}
Let $\Psi_h$ be any composition of (a) sub-flows of strict Hamiltonians ($\partial_s H_i = 0$) and (b) exact damping sub-flows~\eqref{eq:D_constant} or~\eqref{eq:D_statedep} applied for durations $\tau_j$ at frozen positions $x_j$. Then
\[
    \Psi_h^*\,\alpha \;=\; \exp\Bigl(-\sum\nolimits_j \beta(x_j)\,\tau_j\Bigr)\,\alpha \qquad \text{exactly}.
\]
Consequently the cumulative discrete conformal factor $\sigma_h^n$ is an exact quadrature of the conformal rate along the numerical trajectory, and:
\begin{enumerate}
    \item[(i)] for constant damping $\beta \equiv \gamma$, $\sigma_h^n = \gamma\, nh$ exactly, so the discrete envelope rate carries no discretization error at any order;
    \item[(ii)] for state-dependent damping, the $O(h^r)$ term in the conformal tracking estimate~\eqref{eq:disc_conformal_track} is entirely the quadrature-and-trajectory error of sampling $\beta$ along the iterates, verified numerically at the predicted orders in \cref{fig:conformal_scaling}.
\end{enumerate}
\end{prop}

\begin{proof}
Strict sub-flows preserve $\alpha$ exactly, since their conformal rate $\partial_s H_i$ vanishes identically. The damping sub-flow is the \emph{exact} contact flow of $D=\beta(x)s$ for time $\tau_j$. Along this flow $x$ is genuinely frozen ($\nabla_p D = 0$), so its conformal factor is $\exp(-\int_0^{\tau_j}\partial_s D\,dt) = \exp(-\beta(x_j)\tau_j)$ with no approximation. Pullback is multiplicative under composition, which gives the product formula, and (i) and (ii) are immediate specializations.
\end{proof}

\Cref{prop:exact_conformal} makes precise how the contact frame keeps the dissipation bookkeeping under discretization. The stated splitting preserves the contact-form multiplier \emph{exactly}, and for state-dependent damping all $O(h^r)$ losses in that multiplier are pushed into where the damping is sampled. This statement concerns the pullback of $\alpha$. It does not assert exact decay of the original Hamiltonian evaluated at the numerical iterates, which may exhibit an $O(h^r)$ modified-equation ripple. The result is conditional on using exact constituent contact maps whose multipliers have the displayed product. A generic integrator that combines damping with another stage has no automatic exact-multiplier guarantee, although an exact combined contact flow with constant $s$-derivative $\gamma$ still has multiplier $e^{-\gamma\tau}$. \Cref{rem:exact_scope} in \cref{app:proofs} further delineates the role of exact sub-steps and frozen-coefficient damping, and explains why an exact atom with an evolving rate (such as the nonlinear damping~\eqref{eq:D_nonlinear}) still has an exact multiplier, but one that is no longer the sampled sum $\sum_j \beta(x_j)\tau_j$ of the nominal rate, because that rate varies during the sub-step instead of being fixed at the incoming point.

\subsection{Conformal-factor tracking at the predicted order}
\label{sec:tracking_numerics}

For state-dependent damping the discrete conformal factor is an exact quadrature of $\beta$ along the iterates (\cref{prop:exact_conformal}(ii)), so the tracking error against the continuous integral $\int_0^T \beta(x(t))\,dt$ is governed by the order of the splitting, as asserted in \cref{ass:bea} and \cref{lem:mod_conformal}. \Cref{fig:conformal_scaling} verifies this on an anharmonic convex objective $f(x) = \tfrac14 x^4 + \tfrac12 x^2$ with analytic damping $\beta(x) = 0.6 + 0.3\tanh(x^2)$: the cumulative discrete conformal factor converges to the continuous integral at slope $2.01$ for Strang and slope $4.02$ for the Yoshida triple jump, matching the theoretical orders to within fitting error.

\begin{figure}[t]
    \centering
    \includegraphics[width=0.5\textwidth]{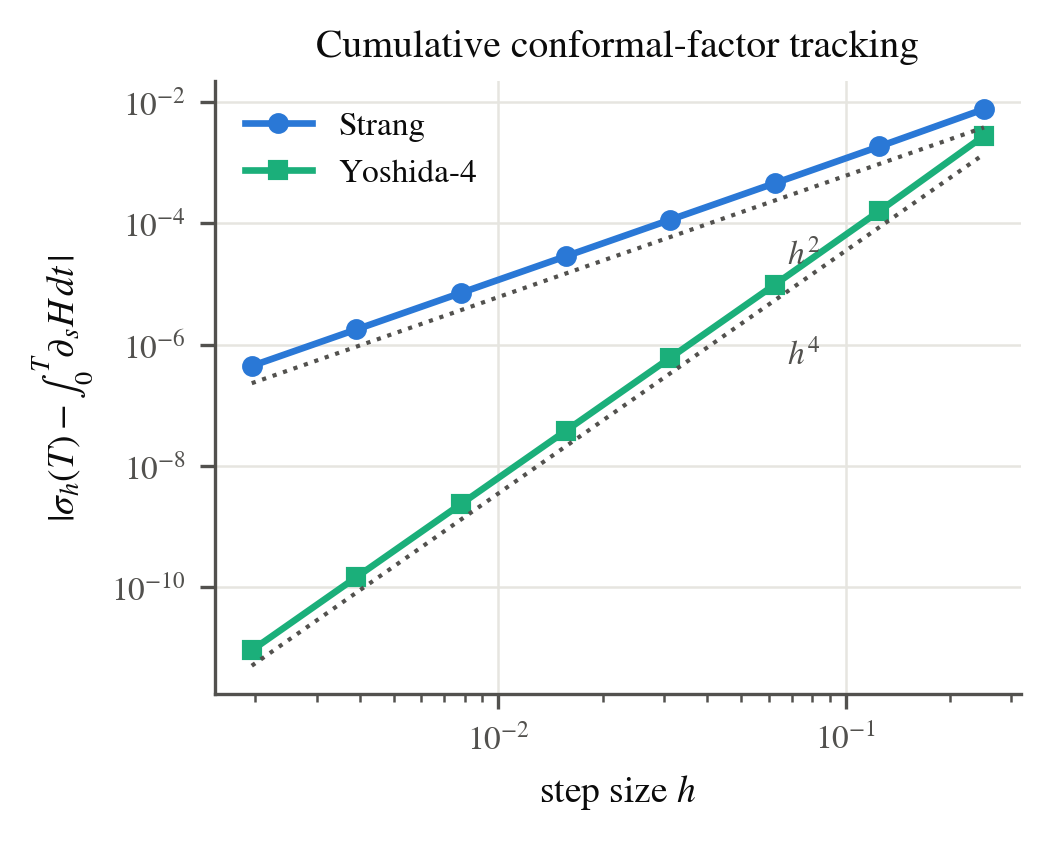}
    \caption{Numerical verification of discrete conformal-factor tracking (\cref{lem:mod_conformal}): error of the cumulative discrete conformal factor $\sigma_h(T)$ against $\int_0^T \partial_s H\,dt$ along the exact flow, for the state-dependent family $H=\tfrac12 p^2 + \tfrac14 x^4 + \tfrac12 x^2 + (0.6+0.3\tanh x^2)\,s$, $T=4$. The dotted guides show the predicted slopes, $h^2$ for Strang (upper pair) and $h^4$ for Yoshida-4 (lower pair), against fitted slopes of $2.01$ and $4.02$.}
    \label{fig:conformal_scaling}
\end{figure}

\subsection{State-dependent damping as a design parameter}
\label{sec:state_dep_design}

With the conformal rate $\partial_s H = \beta(x)$, the damping adapts pointwise to the landscape without learning-rate scheduling or adaptive-gradient preconditioning. Natural schedules include
\begin{align}
    \beta(x) &= \gamma_0 + \gamma_1\,\rho(x)^q &&\text{(objective-adaptive; damp hard while far from optimum)},\label{eq:beta_adapt}\\
    \beta(x) &= \gamma_0 + \gamma_1 / \bigl(1 + c\,\rho(x)^q\bigr) &&\text{(inverse-adaptive; accelerate across plateaus)},\label{eq:beta_inv}\\
    \beta(x) &= \gamma_0 + \gamma_1\,\mathrm{tr}\,\nabla^2 f(x) &&\text{(curvature-adaptive via Hessian trace)},
\end{align}
where $\rho(x)$ is a clipped normalized objective gap (\cref{app:algorithms}), $\gamma_0,\gamma_1,c$ are tuning constants, and the exponent $q$ is a schedule parameter rather than the BEA truncation order used above. Each is realized by the single closed-form sub-flow~\eqref{eq:D_statedep}, so the conformal bookkeeping of \cref{prop:exact_conformal} is exact for the sampled damping along the numerical trajectory. Objective-level guarantees then require one of two additional certificates: the Hamiltonian comparison hypotheses of \cref{cor:state_dep}, or a separate Lyapunov certificate such as \cref{lem:bregman_lyap}, which transfers only through the shadowing statement \cref{cor:aux_transfer}. \Cref{fig:state_damping} makes the conformal mechanism visible on the two-dimensional Rosenbrock function with the objective-adaptive schedule~\eqref{eq:beta_adapt}. Damping is strong in the outer landscape where the objective gap is large, relaxes in the valley, and the certificate $\mcE(z_n)$ tracks its accumulated conformal envelope $e^{-\sigma_n}\mcE(z_0)$ along the entire run. Here $\sigma_n$ denotes the cumulative sampled conformal exponent, and the objective gap stays below the certificate on this trajectory. The schedules~\eqref{eq:beta_adapt} and~\eqref{eq:beta_inv} are exactly the C-Adapt and C-InvGrad algorithms benchmarked in \cref{sec:computation}.

\begin{figure}[t]
    \centering
    \includegraphics[width=0.98\textwidth]{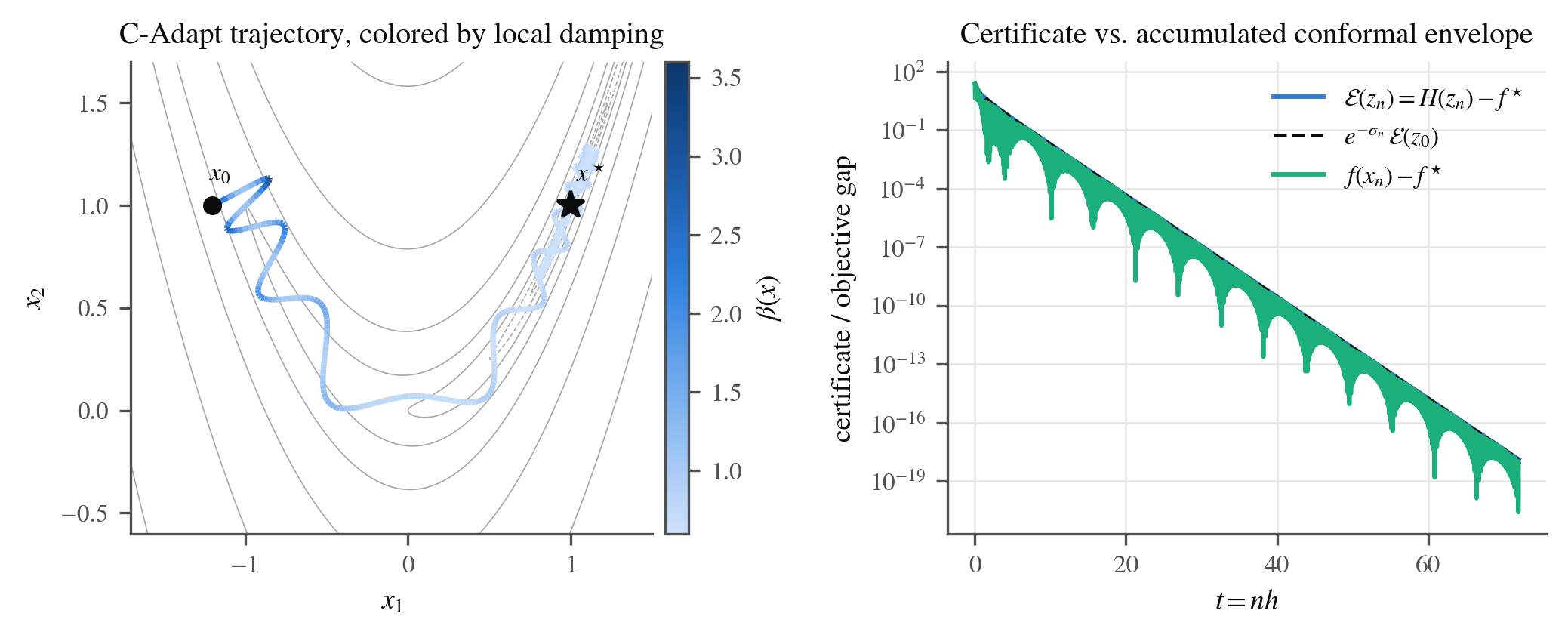}
    \caption{State-dependent damping on 2-D Rosenbrock. \textbf{Left:} Strang trajectory of the objective-adaptive contact optimizer, colored by the local damping $\beta(x)$ (light = weak damping in the valley, dark = strong damping at large objective gap). Contours show $\log_{10} f$. \textbf{Right:} the certificate $\mcE(z_n)=H(z_n)-f^\star$ tracks the accumulated conformal envelope $e^{-\sigma_n}\mcE(z_0)$ (dashed), where $\sigma_n$ is the cumulative sampled conformal exponent, over nineteen orders of magnitude, and the objective gap remains below the certificate, as the comparison hypothesis requires.}
    \label{fig:state_damping}
\end{figure}

For strongly convex objectives, the Lyapunov route can be made fully explicit. The following certificate extends the Polyak heavy-ball construction by an $s^2$ term that absorbs the coupling introduced by $\nabla\beta$. Quadratic certificates of this shape, together with windows restricting the damping relative to the strong-convexity modulus, are standard for constant and time-dependent damping~\cite{polyak1964heavyball,wilson2021lyapunov,attouch2018fast}, and what is new here is the pair of smallness conditions on $\nabla\beta$ that keep the construction valid once the damping varies with the state. Being of non-Hamiltonian form, it transfers to the discrete iterates through the shadowing statement \cref{cor:aux_transfer} rather than through the conformal identity.

\begin{lemma}[Bregman-type Lyapunov certificate for state-dependent damping]\label{lem:bregman_lyap}
Let $f: \R^n \to \R$ be $\mu$-strongly convex, $C^1$, with minimizer $x^\star$ and $f^\star = f(x^\star) = 0$ (the objective normalized at its minimum; see \cref{rem:gauge}). Let $\beta \in C^1$ with $0 < \beta_{\min} \le \beta(x) \le \beta_{\max}$ on a compact region $U$, write $B := \|\nabla\beta\|_{L^\infty(U)}$, and assume the trajectory of $H = \tfrac12\|p\|^2 + f(x) + \beta(x)\,s$ remains in $U$ with $|s(t)| \le S$ for all $t$. Assume the damping and smallness windows
\begin{equation}\label{eq:lyap_windows}
    \textup{(H1)}\ \ \beta_{\min}\,\beta_{\max} \le \tfrac{3}{2}\,\mu,
    \qquad
    \textup{(H2)}\ \ B \le \tfrac{\beta_{\min}}{8},
    \qquad
    \textup{(H3)}\ \ B\,S \le \tfrac{3}{32}\,\beta_{\min}^2 .
\end{equation}
Set $\sigma := \beta_{\min}/2$ and $\rho := \tfrac{2}{3}\,B/\beta_{\min}$ (with $\rho := 0$ if $\nabla\beta \equiv 0$), and define
\[
    \tE(x,p,s) \;:=\; \tfrac12\bigl\|p + \sigma(x - x^\star)\bigr\|^2 \;+\; f(x) \;+\; \tfrac{\rho}{2}\,s^2.
\]
Then on $U$:
\begin{enumerate}
    \item[(i)] $\tE(z) \ge 0$, with equality iff $(x, p, s) = (x^\star, 0, 0)$ when $\rho > 0$ (iff $(x,p) = (x^\star,0)$ when $\rho = 0$);
    \item[(ii)] $\dfrac{\dx}{\dx t}\,\tE \;\le\; -\dfrac{\beta_{\min}}{4}\,\tE$ along the contact flow of $H$;
    \item[(iii)] $f(x) - f^\star \le \tE(z)$, i.e.\ the objective comparison holds with constant $1$ against $\tE$ (the hypothesis $C_W = 1$ of \cref{cor:aux_transfer}, not the Hamiltonian comparison~\eqref{eq:objective_comparison}).
\end{enumerate}
The constants in~\eqref{eq:lyap_windows} are sufficient, not optimized, and $L$-smoothness of $f$ is not needed at this level. In the constant-damping limit $\nabla\beta \equiv 0$ the certificate reduces to the classical Polyak construction, and (H1) is the familiar restriction of the damping below the critical scale $\sqrt\mu$ (in geometric mean).
\end{lemma}

\begin{proof}
See \cref{app:lyapunov}.
\end{proof}

\subsection{Nonlinear and momentum-coupled dissipation}
\label{sec:nonlinear_design}

Two further dissipation atoms are accessible only in the contact frame, in the sense that their conformal factors depend on the state and hence admit no conformal-symplectic realization with a fixed rate.

\paragraph{Self-regulating damping $D=(\gamma/2)s^2$.} The conformal rate $\partial_s H = \gamma s$ depends on the accumulated action variable. Large $|s|$ early in optimization produces strong damping, while $s \to 0$ near the minimum lets the optimizer coast. The sub-flow~\eqref{eq:D_nonlinear} is exact in closed form, and implementations clamp the step below the Riccati singularity at $\tau = 2/(\gamma|s|)$ for $s<0$. Since $D \ge 0$ always, the positivity items of \cref{ass:cert} are structurally free for this family. What is conditional is the \emph{sign} of the rate, since the full flow does not preserve $s\ge0$. In \cref{fig:nonlinear_s2}, the instantaneous rate $\gamma s(t)$ starts large and relaxes as the action variable empties. On a fixed quadratic a well-tuned constant damping still wins. The value of this atom is scheduling-free rate adaptation rather than raw speed, and the comparison is against one tuned baseline on one objective. Atoms in this class chosen with the rate profile in mind, rather than for closed-form convenience, may well be competitive on speed too, which we leave to future work.

\paragraph{Momentum-transfer damping $D = \alpha\|p\|^2 s$.} Here the conformal rate and envelope are
\[
    \partial_s H = \alpha\|p\|^2,
    \qquad
    \exp\!\Bigl(-\alpha\int_0^t \|p\|^2\,d\tau\Bigr).
\]
The damping is self-regulating in the momentum, damping fast motion while leaving slow motion nearly conservative. At the sub-flow level, the closed form~\eqref{eq:D_psq} shows the momentum norm decreasing while the action variable grows by the reciprocal factor, an explicit transfer between the two dissipation channels.

\begin{figure}[t]
    \centering
    \includegraphics[width=0.98\textwidth]{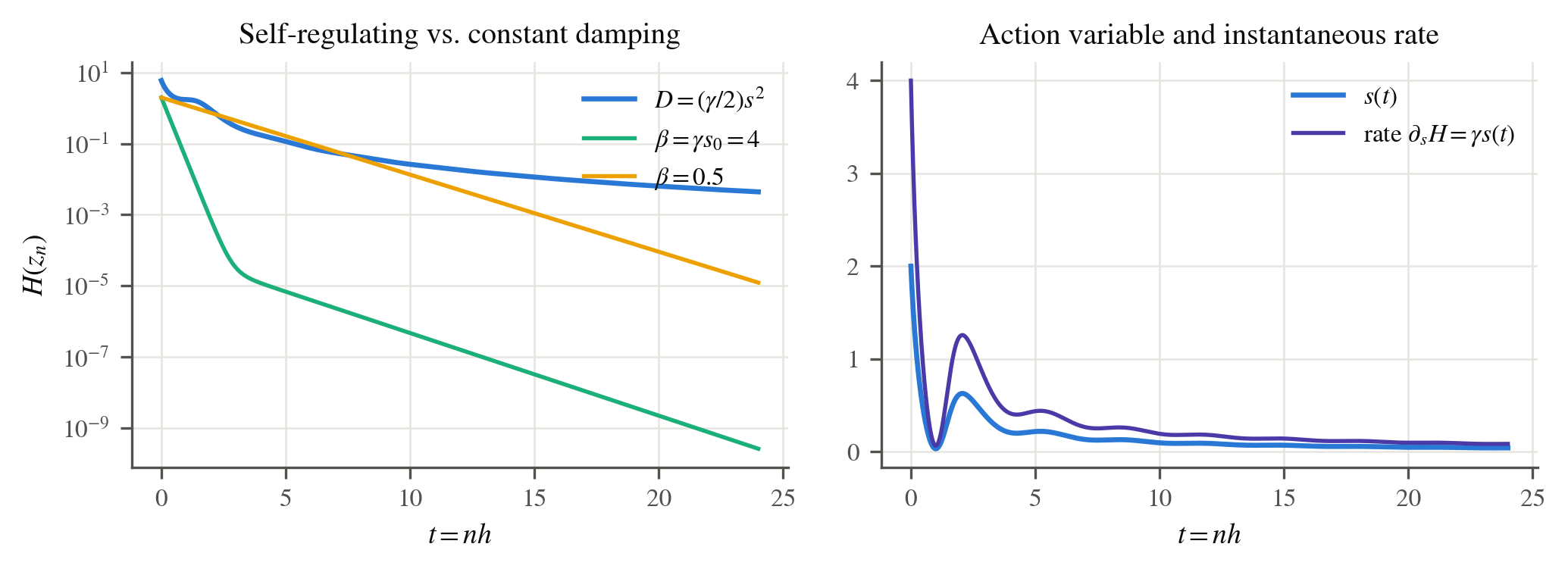}
    \caption{Self-regulating nonlinear dissipation $D=(\gamma/2)s^2$ on a quadratic ($\lambda=1$, $\gamma=2$, $s_0=2$). \textbf{Left:} Hamiltonian decay vs.\ two constant-damping references. The nonlinear atom is aggressive early and coasts late, without any schedule. \textbf{Right:} the action variable $s(t)$ and the instantaneous conformal rate $\partial_s H = \gamma s(t)$.}
    \label{fig:nonlinear_s2}
\end{figure}

\section{Computation}
\label{sec:computation}

The experiments below play two distinct roles. The deterministic smooth benchmarks are the closest to \cref{thm:rate}. They instantiate autonomous contact splittings (C-HB, C-Adapt, C-InvGrad) in settings where the smoothness and compact-region hypotheses of \cref{ass:reg} are plausible and the certificate mechanisms of \cref{sec:closed_form} are active, although we do not verify \cref{ass:cert} on these nonconvex landscapes, so the runs should be read as instantiations rather than certified deployments. The deep-learning experiments (Contact-Adam and contact-SGD variants with stochastic mini-batches, and the iteration-scheduled C-NAG damping of \cref{tab:classical_algos_td}) are \emph{stress tests of the contact design template}. The theorem covers neither stochastic gradients nor nonautonomous schedules, and these results are evidence that the template remains competitive under practical distortions, not direct validations of the certificate-transfer theory. Structure-preserving pre-symplectic, conformal-symplectic, and contact integrators have shown strong behavior on optimization tasks before~\cite{francca2020conformal,francca2021dissipative,vermeeren2019contact,bravetti2020numerical}. Algorithmic details, hyperparameter selection, and extended ablations are collected in \cref{app:algorithms,app:extended_numerics}.

\subsection{Deterministic Benchmarks}

We compare classical and contact Hamiltonian optimization algorithms on two standard deterministic benchmarks: the generalized Rosenbrock function~\cite{rosenbrock1960automatic} and the Wood function from the MGH test set~\cite{more1981testing}. The non-contact baselines are Gradient Descent (GD), Nesterov's accelerated gradient (NAG), and Relativistic gradient descent (RGD~\cite{francca2021dissipative}). The contact methods are C-HB, C-NAG, C-Adapt, and C-InvGrad, described in \cref{app:algorithms}. The figures and \cref{tab:deterministic_benchmarks} report the main comparison, with robustness studies deferred to \cref{app:extended_numerics}.

\paragraph{Rosenbrock.} We consider the optimization of the generalized Rosenbrock function
\[
    f(x) = \sum_{i=1}^{d-1} \Bigl[100(x_{i+1} - x_i^2)^2 + (1-x_i)^2\Bigr]
\]
with $d=100$. On this ill-conditioned chain objective, the state-adaptive contact methods converge faster (${\sim}5$K iterations) than the constant-damping contact baselines (${\sim}12$K iterations), GD, and NAG, visualized in \cref{fig:rosenbrock}. The geometric RGD baseline is competitive with the contact methods at the tuned starting point (\cref{tab:deterministic_benchmarks}). The separation between them appears in the fixed-tuning initial-condition robustness study of \cref{fig:initial_condition_robustness}.

\begin{figure}[ht]
    \centering
    \includegraphics[width=0.9\textwidth]{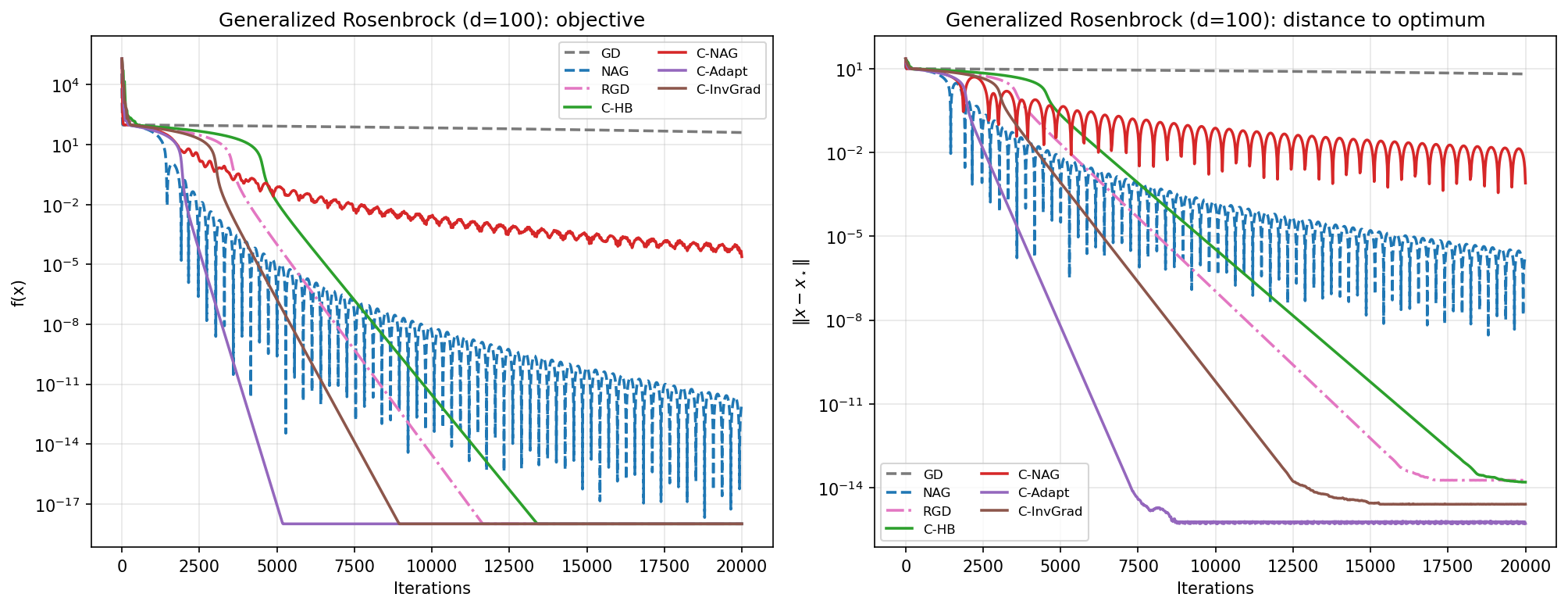}
    \caption{Comparison of contact Hamiltonian algorithms with constant damping and state-adaptive damping on the Rosenbrock function. The state-adaptive damping algorithm converges significantly faster than the constant damping algorithm.}
    \label{fig:rosenbrock}
\end{figure}

\paragraph{Wood function.} We consider the standard four-variable Wood function
\[
    \begin{aligned}
        f(x) ={}& 100(x_1^2 - x_2)^2 + (x_1 - 1)^2 + (x_3 - 1)^2 + 90(x_3^2 - x_4)^2 \\
        &+ 10.1\qty[(x_2 - 1)^2 + (x_4 - 1)^2] + 19.8(x_2 - 1)(x_4 - 1),
    \end{aligned}
\]
so $d=4$. This standard nonconvex test problem gives a low-dimensional check that the same state-dependent damping mechanism remains effective away from the Rosenbrock geometry. Once again, state-adaptive methods and RGD converge faster (${\sim}5$K iterations) than GD or NAG (${\sim}17$K iterations) whose distance to the minimizer remains above machine precision, visualized in \cref{fig:wood}.

\begin{figure}[ht]
    \centering
    \includegraphics[width=0.9\textwidth]{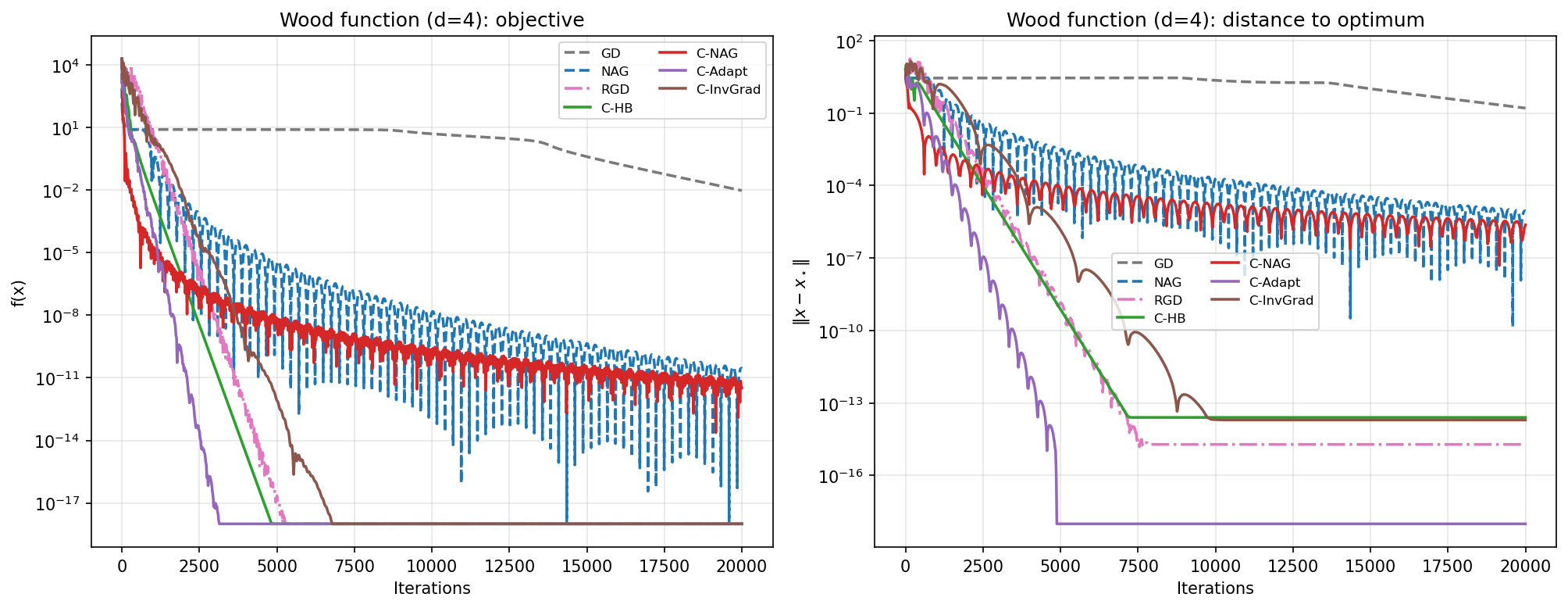}
    \caption{Comparison of contact Hamiltonian algorithms with constant damping and state-adaptive damping on the Wood function. The state-adaptive damping algorithm converges significantly faster than the constant damping algorithm.}
    \label{fig:wood}
\end{figure}

\begin{table}[ht]
    \centering
    \small
    \renewcommand{\arraystretch}{1.15}
    \begin{tabular}{lcccc}
        \toprule
        & \multicolumn{2}{c}{Generalized Rosenbrock ($20000$ iterations)} & \multicolumn{2}{c}{Wood ($20000$ iterations)} \\
        Method & $f_T$ & $\|x_T-x_\star\|$ & $f_T$ & $\|x_T-x_\star\|$ \\
        \midrule
        GD & $4.05 \times 10^{1}$ & $6.41$ & $9.34 \times 10^{-3}$ & $1.61 \times 10^{-1}$ \\
        NAG & $1.05 \times 10^{-12}$ & $2.05 \times 10^{-6}$ & $2.75 \times 10^{-11}$ & $8.74 \times 10^{-6}$ \\
        RGD & $1.21 \times 10^{-28}$ & $1.86 \times 10^{-14}$ & $1.84 \times 10^{-30}$ & $1.94 \times 10^{-15}$ \\
        \midrule
        C-HB & $7.99 \times 10^{-29}$ & $1.58 \times 10^{-14}$ & $2.38 \times 10^{-28}$ & $2.54 \times 10^{-14}$ \\
        C-NAG & $2.43 \times 10^{-5}$ & $8.02 \times 10^{-4}$ & $3.28 \times 10^{-12}$ & $2.36 \times 10^{-6}$ \\
        C-Adapt & $7.47 \times 10^{-30}$ & $5.21 \times 10^{-16}$ & $0$ & $0$ \\
        C-InvGrad & $7.85 \times 10^{-30}$ & $2.57 \times 10^{-15}$ & $1.48 \times 10^{-28}$ & $1.99 \times 10^{-14}$ \\
        \bottomrule
    \end{tabular}
    \caption{Final deterministic-benchmark statistics. Entries report the final objective and the final distance to the known minimizer after $20000$ iterations on generalized Rosenbrock ($d=100$) and Wood ($d=4$). All tabulated values, including the printed zeros, should be read as floating-point residuals rather than exact symbolic convergence certificates.}
    \label{tab:deterministic_benchmarks}
\end{table}

All entries in \cref{tab:deterministic_benchmarks} should be interpreted at floating-point accuracy. The sub-$10^{-12}$ values and printed zeros indicate that the iterates have reached the numerical minimizer in machine precision, not that the discrete method carries an exact symbolic zero-error certificate. The tuned hyperparameters, Hamiltonian ablations, and robustness studies are reported in \cref{app:algorithms,app:extended_numerics}. The fixed-tuning robustness study in \cref{fig:initial_condition_robustness} is especially stark. Over $100$ Rosenbrock-100 random sign starts, with parameters tuned at the original starting point, C-InvGrad reaches numerical precision on at least $90\%$ of starts, and C-Adapt has a machine-precision median with a small outlier tail. Tuned RGD, in contrast, has median objective gap $7.8\times10^{5}$ and median distance $14.5$.

\subsection{Deep Learning}
We next evaluate the contact variants in two learning regimes that stress different parts of an optimizer. The first is a full-batch physics-informed neural network (PINN) for the viscous Burgers equation~\cite{burgers1948mathematical,raissi2019physics}, where the loss couples PDE residual, initial condition, and boundary terms. The second is mini-batch training of a ResNet-18 architecture~\cite{he2016deep} on CIFAR-10~\cite{krizhevsky2009learning}, a standard image-classification benchmark with stochastic gradients. Together these tasks test whether the contact modifications remain competitive beyond low-dimensional deterministic objectives. We summarize the final deep-learning statistics in \cref{tab:deep_learning_summary}.

\paragraph{Burgers PINN.} The Burgers equation is given by the PDE $\partial_t u + u \partial_x u = \nu \partial_{xx} u$ for viscosity $\nu > 0$. We train a PINN on the domain $[0,1] \times [0,1]$ with periodic boundary conditions and initial condition $u(x,0) = -\sin(\pi x)$. The final full-batch comparison uses Adam, AdamW, SGD with momentum, L-BFGS, and two contact-SGD variants. In \cref{fig:deep_learning} we report the full PINN loss against both cumulative gradient evaluations and wall-clock time, so that multi-evaluation methods such as leapfrog splitting and L-BFGS line search are compared on fair computational axes. We observe that contact methods behave more smoothly compared to Adam and AdamW, and, while less competitive than the second-order L-BFGS, the contact variants reach losses of $3\times 10^{-4}$ compared to $10^{-3}$ for classical Adam variants, for similar gradient evaluation budgets. The experimental setup and hyperparameter-selection details are reported in \cref{app:algorithms}.

\paragraph{ResNet-18 on CIFAR-10.} CIFAR-10 contains 60,000 color images across 10 classes, and ResNet-18 provides a compact residual-network test case for mini-batch image classification. We compare the contact-Adam variants with Adam, AdamW, and SGD with momentum over 200 epochs. In the right panel of \cref{fig:deep_learning}, AdamW reaches high test accuracy early, while SGD+M and the contact variants finish higher in the reported run over three seeds, with final test accuracies around $95\%$. This gives a complementary stochastic-gradient robustness check of the proposed contact algorithms. Experimental setup and hyperparameter-selection details are reported in \cref{app:algorithms}.

\begin{figure}
    \centering
    \begin{subfigure}[t]{0.32\textwidth}
        \centering
        \includegraphics[width=\linewidth]{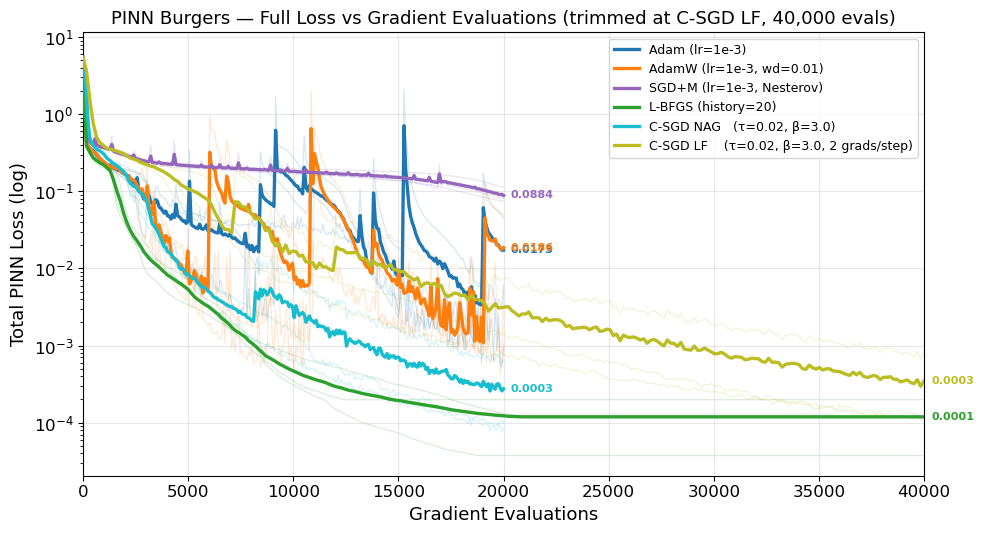}
        \caption{Burgers PINN full loss versus gradient evaluations.}
    \end{subfigure}\hfill
    \begin{subfigure}[t]{0.32\textwidth}
        \centering
        \includegraphics[width=\linewidth]{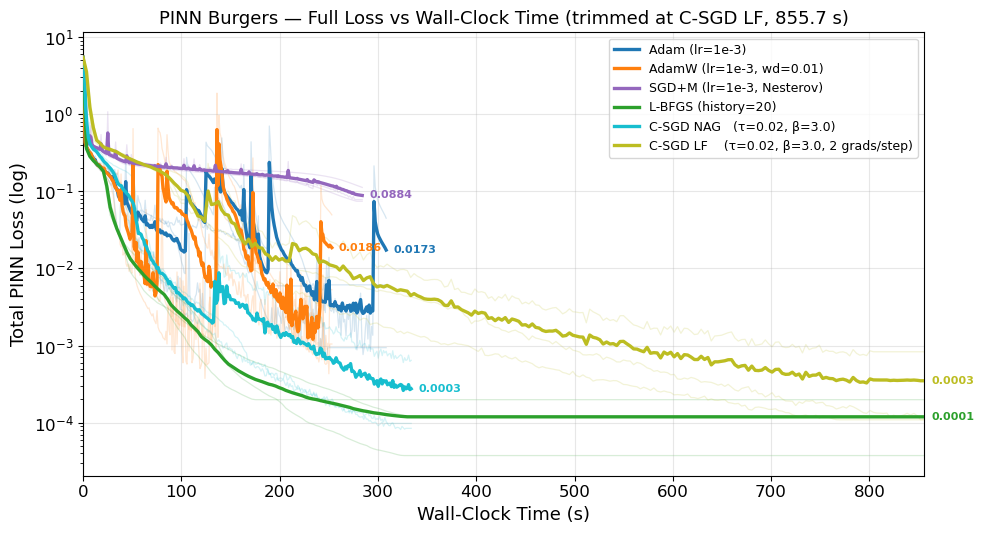}
        \caption{Burgers PINN full loss versus wall-clock time.}
    \end{subfigure}\hfill
    \begin{subfigure}[t]{0.32\textwidth}
        \centering
        \includegraphics[width=\linewidth]{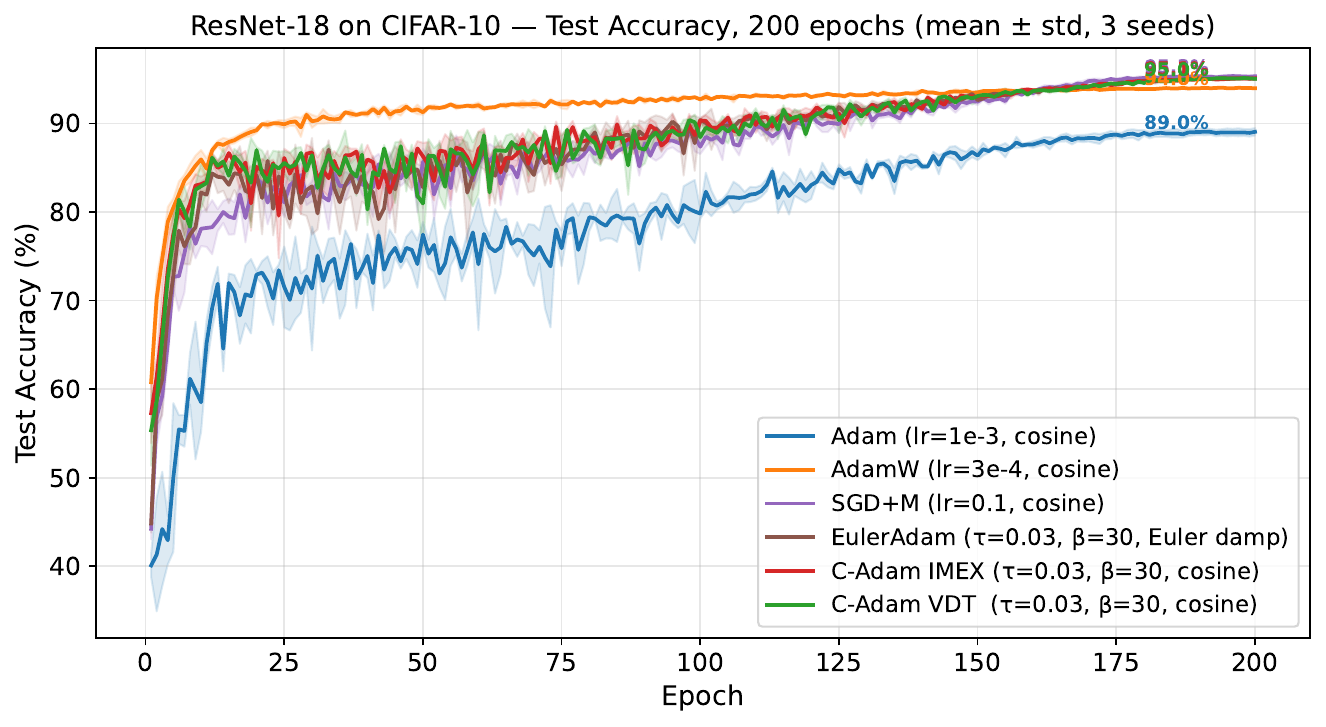}
        \caption{ResNet-18 on CIFAR-10 test accuracy.}
    \end{subfigure}
    \caption{Deep-learning optimization comparisons across full-batch and mini-batch regimes. On the Burgers PINN task we report the full objective against both gradient evaluations and wall-clock time, while the CIFAR-10 panel shows test accuracy over training.}
    \label{fig:deep_learning}
\end{figure}

\begin{table}[t]
    \centering
    \small
    \renewcommand{\arraystretch}{1.15}
    \begin{tabular}{lccccc}
        \toprule
        & \multicolumn{3}{c}{Burgers PINN} & \multicolumn{2}{c}{ResNet-18 on CIFAR-10} \\
        \cmidrule(lr){2-4} \cmidrule(lr){5-6}
        Method & Best loss & Final loss & Grad evals & Best acc. & Final acc. \\
        \midrule
        Adam & $1.22 \times 10^{-3}$ & $1.73 \times 10^{-2}$ & 20,000 & 89.12\% & 89.05\% \\
        AdamW & $6.87 \times 10^{-4}$ & $1.86 \times 10^{-2}$ & 20,000 & 94.06\% & 93.98\% \\
        SGD+M & $8.81 \times 10^{-2}$ & $8.84 \times 10^{-2}$ & 20,000 & 95.40\% & 95.33\% \\
        L-BFGS & $1.19 \times 10^{-4}$ & $1.19 \times 10^{-4}$ & 123,040 & n/a & n/a \\
        C-SGD NAG & $2.04 \times 10^{-4}$ & $2.75 \times 10^{-4}$ & 20,000 & n/a & n/a \\
        C-SGD LF & $2.62 \times 10^{-4}$ & $3.49 \times 10^{-4}$ & 40,000 & n/a & n/a \\
        C-Adam IMEX & n/a & n/a & n/a & 95.11\% & 95.08\% \\
        C-Adam VDT & n/a & n/a & n/a & 95.15\% & 95.04\% \\
        \bottomrule
    \end{tabular}
    \caption{Final deep-learning benchmark summary. Burgers PINN entries report mean best loss, mean final loss, and total gradient evaluations for the full-budget run; CIFAR-10 entries report mean best and final test accuracy over three seeds. Entries marked n/a indicate methods not included in that task's final comparison roster.}
    \label{tab:deep_learning_summary}
\end{table}

\section{Discussion}
\label{sec:Discussion}
This paper develops the contact Hamiltonian formalism as a framework for geometric optimization~\cite{bravetti2023bregman,vermeeren2019contact}, allowing both analysis of existing algorithms and a template for designing new dissipative methods. The main result is not a blanket global acceleration theorem, but a Hamiltonian-certificate transfer principle stated under three named, independently checkable hypotheses. Whenever $H-H^\star$ controls the objective gap on a compact region (\cref{ass:cert}), an order-$r$ contact splitting transfers the corresponding finite-horizon decay envelope up to the modified-conformal-factor perturbation and BEA defect (\cref{thm:rate}). The quadratic heavy-ball analysis is a fully solvable demonstration of this machinery. Its projected dissipative-leapfrog spectrum agrees with the conformal-symplectic literature~\cite{francca2020conformal,francca2021dissipative}, while the augmented contact Hamiltonian supplies the sharp objective comparison needed to verify our certificate hypotheses (\cref{prop:quadratic}). For state-dependent damping on strongly convex objectives, the explicit non-Hamiltonian Lyapunov construction of \cref{lem:bregman_lyap} instead transfers through the auxiliary-shadowing result \cref{cor:aux_transfer}. In addition, the discrete contact-form multiplier of the chosen master splitting is exact (\cref{prop:exact_conformal}), so state-dependent damping error enters through quadrature and trajectory sampling rather than a defect in the constituent contact maps.

Within that scope, the same formalism acts as a template for designing optimization algorithms with state-dependent, nonlinear, and momentum-coupled damping, while leaving room for useful nonseparable contact Hamiltonians beyond the particular $H=K+V+D$ template emphasized here. This viewpoint is complementary to recent ODE and structure-preserving perspectives on optimization~\cite{su2016differential,wibisono2016variational,francca2020conformal,francca2021dissipative,maddison2018hamiltonian_descent}. Rather than embedding dissipation into an enlarged symplectic system, we use the intrinsic contact conformal factor as the geometric object that links continuous and discrete rate certificates.

\paragraph{Limitations.} The theorem is a finite-horizon, compact-region statement, and its objective-level conclusion is conditional on the Hamiltonian comparison hypothesis. Outside the quadratic family verified here, supplying that comparison remains problem-specific work. Auxiliary Lyapunov functionals provide a second finite-horizon shadowing route but do not inherit the conformal decay identity. As in other backward-error analyses of structure-preserving discretizations, the result controls a modified-flow certificate over the time window and should not be read as a separate fixed-point or minimizer-preservation theorem for the implemented optimizer. Time-dependent damping laws are covered geometrically through the autonomous lift of \cref{prop:lift}, but its certificate hypotheses require separate verification and its compactness region must remain away from the $t=0$ singularity. The deep-learning experiments involve stochastic gradients and adaptive-moment estimates that the deterministic theory does not certify.

\paragraph{Future directions.} We see three natural extensions: (1) determining optimal choices of kinetic and dissipation atoms for particular objective classes, for which the closed-form quadratic analysis of \cref{sec:quadratic} provides the solvable base case; (2) extending the framework to stochastic optimization algorithms, which requires a stochastic contact Hamiltonian formalism that is still in its infancy in the mathematical physics literature \cite{wei2021formulation,zhan2023numerical}; and (3) Bregman and mirror-descent kinetics via contact changes of variables~\cite{bravetti2023bregman}, which fit the present framework as a coordinate change before splitting.

\section{Acknowledgements}
I am sincerely thankful to my advisors, Soledad Villar and Mauro Maggioni, at Johns Hopkins University for their guidance and support throughout this work. I am grateful to Ian McPherson who originally introduced me to the general geometric optimization literature, and Ben Grimmer and Ren\'{e} Vidal for insightful discussions on geometric optimization and pointing out relevant literature I would have otherwise missed. Any errors are my own.

\bibliography{main}
\bibliographystyle{unsrtnat}

\newpage

\appendix

\section{Contact Hamiltonian Dynamics}
\label{app:contact_dynamics}

For standard references on contact geometry and contact Hamiltonian dynamics, see \cite{arnold2013mechanics,krieg_michorl1997convenient,banyaga2013structure,guenther1996herglotz,olver1995equivalence}. We give a brief review of the relevant definitions and properties here for completeness.

A contact manifold is a smooth $(2n+1)$-dimensional manifold $\mcM$ equipped with a contact form $\alpha$ satisfying the non-degeneracy condition $\alpha \wedge (d\alpha)^n \neq 0$. The contact form $\alpha$ defines a hyperplane distribution $\xi = \ker \alpha$ and a Reeb vector field $R$ satisfying $\alpha(R) = 1$ and $d\alpha(R, \cdot) = 0$. Given a contact Hamiltonian $H: \mcM \to \R$, and using the sign convention of the main text, the associated contact Hamiltonian vector field $X_H$ is defined by the equations
\begin{equation}
    \alpha(X_H) = -H, \qquad d\alpha(X_H, \cdot) = dH - (R(H))\alpha.
\end{equation}
The flow generated by $X_H$ is a contactomorphism, i.e. it preserves the contact structure $\xi$ and the contact form $\alpha$ up to a conformal factor,
\begin{equation}
    (\Phi_{X_H}^t)^*\alpha
    = \exp\!\Bigl(-\int_0^t R(H)(\Phi_{X_H}^\tau z)\,d\tau\Bigr)\alpha.
\end{equation}

Darboux's theorem guarantees that locally, one can always find coordinates $(x,p,s)$ on $J^1(\R^n)$ such that the contact Hamiltonian equations take the form
\begin{equation}
    \dot{x} = \nabla_p H, \qquad \dot{p} = -\nabla_x H - p \frac{\partial H}{\partial s}, \qquad \dot{s} = p^T \nabla_p H - H.
\end{equation}
The associated contact form in Darboux coordinates is given by $\alpha = ds - p^T dx$, and the Reeb field is $R=\partial_s$. Thus the infinitesimal conformal rate of the contact form is $-\partial_s H$, while the Hamiltonian itself satisfies $\dot H=-(\partial_sH)H$. For $H=\gamma s$ with $\gamma>0$ one obtains $\dot H=-\gamma H$ and $(\Phi_{X_H}^t)^*\alpha=e^{-\gamma t}\alpha$.

Jet spaces have a natural contact structure given by the contact form $\alpha = ds - p^T dx$, and thus any smooth function $H: J^1(\R^n) \to \R$ defines a contact Hamiltonian system globally on $J^1(\R^n)$ with the associated equations of motion given by the above equations. This is the setting we work with in this paper, which is natural for optimization problems where the parameters are given by $x$, the momentum is given by $p$, and the auxiliary `dissipation' or `acceleration' variable is given by $s$. The jet space $J^1(\R^n)$ can be thought of as an extended phase space for optimization, where the dynamics of the parameters $x$ are coupled to the dynamics of the momentum $p$ and the auxiliary variable $s$ through the contact Hamiltonian equations.

The finite-dimensional construction is not limited to parameters in $\R^n$. On a smooth parameter manifold one may work locally in contact Darboux charts, with global formulations subject to the topology and chosen contact bundle. Extending the framework to infinite-dimensional function spaces requires a separate functional-analytic contact structure, and a Darboux theorem is not automatic in that setting, so such extensions lie beyond the scope of the present paper.

\section{Sub-flow Catalogue: Details}
\label{app:catalogue}

This appendix records two derivations deferred from \cref{sec:subflows}.

\paragraph{The incoming-point frozen shear is not an exact contact map.} The variable-metric Hamiltonian $K(x,p)=\tfrac12p^\top M(x)^{-1}p$ itself presents no geometric obstruction: because $\partial_sK=0$, its exact geodesic flow~\eqref{eq:K_metric_geodesic} is an autonomous strict contact flow. The distinction is computational: that flow is generally not the explicit shear~\eqref{eq:K_metric}. If $M=M(x)$ is evaluated at the incoming point of each step and then inserted into that shear, as in the implemented Contact-Newton variant with $M=\nabla^2f(x)$, a direct computation gives
\[
    \psi^*\alpha \;=\; \alpha \;-\; \tfrac{\tau}{2}\,p^\top \partial_{x_j}\!\bigl(M(x)^{-1}\bigr)\,p\;dx_j ,
\]
an $O(\tau\,\|\partial_x M\|)$ defect. The extra term is not proportional to $\alpha$, so this particular sub-step is not conformal for \emph{any} factor. For this reason the implemented variant is identified as framework-motivated in \cref{tab:classical_algos} rather than as an instance of the exact-contact hypotheses. When the preconditioner is instead an external state frozen during the step (the Adam EMA buffer, the L-BFGS pair history), no defect arises. The map is an exact strict contactomorphism at each step, and only the nonautonomy of the step sequence must be tracked separately, as noted in \cref{sec:subflows}. Alternatively, an exact geodesic solve or a symplectic discretization equipped with its generating-function contact lift avoids this defect while retaining the autonomous variable-metric Hamiltonian; analyzing such an inner approximation requires extending the exact-sub-flow BCH argument of \cref{app:bch}.

\paragraph{Momentum-dependent damping: derivation of the closed form~\eqref{eq:D_psq}.} For $D_{p} = \alpha\|p\|^2 s$, polynomial in $p$ of degree $2$ with nontrivial $s$-dependence, the contact Hamiltonian equations~\eqref{eq:contact_eom} read
\[
    \dot x = 2\alpha s\,p,\qquad
    \dot p = -\alpha\|p\|^2\,p,\qquad
    \dot s = \alpha\|p\|^2\,s.
\]
Two conserved quantities are immediate: $p_i\,s$ is constant for each $i$, since $\frac{\dx}{\dx t}(p_i s) = \dot p_i s + p_i \dot s = 0$, and $w = \|p\|^2$ satisfies the Riccati equation $\dot w = -2\alpha w^2$, solved by $w(\tau) = w_0/(1+2\alpha w_0\tau)$. Setting $c = 1 + 2\alpha w_0\tau$ gives the momentum and action updates $p \mapsto p/\sqrt c$ and $s \mapsto s\sqrt c$ of~\eqref{eq:D_psq}. The $x$-update is linear in $\tau$ because $\dot x = 2\alpha s p$ is constant along the sub-flow (each $p_i s$ is conserved).

\section{\texorpdfstring{Backward Error Analysis for Master $H$ Splittings}{Backward Error Analysis for Master H Splittings}}
\label{app:bch}

This appendix collects the formal backward-error input used in \cref{thm:rate} (packaged as \cref{ass:bea}), specialized to the master Hamiltonian $H = K(p) + V(x) + d(x)\,s$ with the splitting sub-flows of \cref{sec:subflows}. The rate-transfer theorem uses only this specialized BEA input. The broader contact Lie-algebra calculus, including the strict/prolonged generator classes and the local universality theorem used to motivate the design catalogue, is developed in~\cite{kevrekidis2026local}.

Contact Hamiltonians form a Lie algebra under the contact-Jacobi bracket $\qty{\cdot,\cdot}_c$,
\begin{equation}\label{eq:cjbracket}
    \qty{A,B}_c \;=\; \{A,B\} + A\,\partial_s B - B\,\partial_s A + p\,(\partial_s A\, \partial_p B
- \partial_s B\, \partial_p A),
\end{equation}
where $\{\cdot,\cdot\}$ is the standard Poisson bracket on $T^*\R^n$. With the vector-field convention of \eqref{eq:contact_eom}, this bracket satisfies $X_{\qty{A,B}_c}=-[X_A,X_B]$. Composition of contact flows obeys the corresponding Baker--Campbell--Hausdorff (BCH) formula, which yields the modified-Hamiltonian generator of a splitting integrator.

\begin{lemma}[Contact BCH formal series]\label{lem:contact_bch}
Let $H_1,\dots,H_m \in C^\infty(J^1(\R^n))$ and let $\Phi_{H_i}^h$ denote the time-$h$ contact flow of $H_i$. Then there exists a formal contact Hamiltonian series $\tH(h) \in C^\infty(J^1(\R^n))[[h]]$ such that
\begin{equation}\label{eq:bch_formal}
    \Phi_{H_m}^{a_m h} \circ \cdots \circ \Phi_{H_1}^{a_1 h}
    \;=\; \exp\!\bigl(h\,\tH(h)\bigr)
\end{equation}
formally in $h$, where the coefficients of $\tH(h)$ are explicit polynomials in nested contact-Jacobi brackets~\eqref{eq:cjbracket} of $H_1,\dots,H_m$ with rational coefficients in the $a_i$. For an order-$r$ symmetric composition (in the sense of cancelling all terms of degree $<r$ in $h$), one has $\tH(h) = H + h^r H_r + O(h^{r+2})$ where $H = \sum_i a_i H_i$ at the BCH-leading order and $H_r$ is a fixed polynomial in nested $\qty{\cdot,\cdot}_c$-brackets of the $H_i$.
\end{lemma}

The lemma is the contact specialization of the abstract Lie-algebraic BCH theorem~\cite{hairer2006geometric,reich1999backward}, since any Lie-bracket BCH expansion applies verbatim once the Lie bracket is fixed, and~\eqref{eq:cjbracket} verifies the Jacobi identity.

\begin{exmp}[Strang on the master $H$]\label{ex:strang_bch}
For the master Hamiltonian~\eqref{eq:master} and the Strang composition~\eqref{eq:strang}, \cref{lem:contact_bch} gives, through second order,
\begin{equation}\label{eq:strang_modH}
    \begin{aligned}
    \tH(h) \;=\; H \;+\; h^2\Big(&
        -\frac{1}{24}\qty{V,\qty{V,K}_c}_c
        +\frac{1}{12}\qty{K,\qty{K,V}_c}_c \\
        &-\frac{1}{24}\qty{D,\qty{D,K+V}_c}_c
        +\frac{1}{12}\qty{K+V,\qty{K+V,D}_c}_c\Big)
        \;+\; O(h^4).
    \end{aligned}
\end{equation}
The first line is the correction from the inner Verlet composition, while the second comes from the outer damping composition. For $D=\gamma s$, every bracket correction beyond the leading $\gamma s$ is strict: if $G$ is independent of $s$, then
\[
    \qty{\gamma s,G}_c=\gamma\bigl(p^T\nabla_pG-G\bigr),
\]
which is again independent of $s$, and strict Hamiltonians are closed under the contact-Jacobi bracket. Hence the formal modified Hamiltonian satisfies $\partial_s\tH(h)=\gamma$ to all orders. This agrees with the exact multiplier in \cref{prop:exact_conformal}. The modified Hamiltonian itself, and therefore its values along the numerical orbit, can still differ from the original $H$ by $O(h^2)$.
\end{exmp}

The formal series~\eqref{eq:bch_formal} need not converge. For finitely smooth data, one truncates it at a fixed order and obtains the usual algebraic backward-error defect. Analyticity permits optimal truncation at an order $N \asymp 1/h$, upgrading the same defect to one that is exponentially small in $1/h$.

\begin{lemma}[Smooth and analytic BEA shadowing]\label{lem:smooth_analytic_bea}
Let $U \subset J^1(\R^n)$ be compact, let $H = K(p) + V(x) + d(x)\,s$, and let $\Psi^{(r)}_h$ be an order-$r$ contact splitting integrator for $H$ obtained by symmetric composition of the sub-flows of \cref{sec:subflows}. Fix a BEA truncation order $q\ge r$ and assume that $K,V,d$ and the sub-flows have enough bounded $C^{q+2}$ norms on a neighborhood of $U$ to carry out the BCH expansion through order $q$. Then for sufficiently small $h>0$ there exist constants $C_T>0$ and a modified contact Hamiltonian $\tH_{h,q}$ on a smaller compact $U' \subset U$ satisfying
\begin{equation}\label{eq:smooth_bea}
    \tH_{h,q} \;=\; H + O(h^r) \qquad \text{uniformly in } C^1(U'),
\end{equation}
and such that the iterates $z_n = (\Psi^{(r)}_h)^n z_0$ remain in $U'$ on $[0,T]$ and obey the finite-horizon shadowing estimate
\begin{equation}\label{eq:shadow}
    \bigl\| (\Psi^{(r)}_h)^n - \Phi^{nh}_{\tH_{h,q}} \bigr\|_{C^1(U')}
    \;\le\; C_T\,h^q,
    \qquad nh \le T.
\end{equation}
If, in addition, $K,V,d$ and the sub-flows are real analytic and admit holomorphic extensions to a complex neighborhood of $U$, then the truncation may be chosen optimally and the estimate improves to
\begin{equation}\label{eq:analytic_shadow}
    \bigl\| (\Psi^{(r)}_h)^n - \Phi^{nh}_{\tH_{h,N(h)}} \bigr\|_{C^1(U')}
    \;\le\; C_T\,e^{-c/h},
    \qquad nh \le T,
\end{equation}
for constants $C_T,c>0$ independent of $h$ and $n$.
\end{lemma}

\noindent\emph{Proof sketch.} Apply the Lie-algebraic BCH series of \cref{lem:contact_bch} term-by-term. With only finite smoothness, truncate after the fixed order $q$ supported by the available derivatives. The residual local defect is algebraic, and a Gr\"onwall comparison over $nh\le T$ gives~\eqref{eq:shadow}. With analytic data, the coefficients are bounded by the radius of holomorphic extension, so one can truncate at the optimal index $N\asymp 1/h$ and obtain the standard $e^{-c/h}$ tail~\eqref{eq:analytic_shadow}. The argument is the contact-Jacobi-bracket version of Hairer--Lubich--Wanner symplectic BEA~\cite{hairer2006geometric}. \hfill$\square$

The relevant consequence for optimization is that the conformal factor of the modified Hamiltonian tracks that of the original $H$, both pointwise and in cumulative integrated form along the iterates.

\begin{lemma}[Modified conformal factor and discrete tracking]\label{lem:mod_conformal}
Under the hypotheses of \cref{lem:smooth_analytic_bea}, the modified conformal factor $\tlam_{h,q} := \partial_s \tH_{h,q}$ satisfies
\begin{equation}\label{eq:mod_conformal_uniform}
    \tlam_{h,q} \;=\; \partial_s H + O(h^r)
    \qquad \text{uniformly on } U'.
\end{equation}
Moreover, write $\lambda_{\Psi}$ for the conformal exponent of $\Psi^{(r)}_h$, the scalar function defined by $(\Psi^{(r)}_h)^*\alpha = e^{\lambda_{\Psi}}\,\alpha$ as in \cref{sec:contact_opti}. The cumulative discrete conformal factor of $\Psi^{(r)}_h$, defined as
\[
    \sigma_h^n(z) := -\sum_{k=0}^{n-1} \lambda_{\Psi}\bigl((\Psi^{(r)}_h)^k z\bigr),
\]
satisfies
\begin{equation}\label{eq:disc_conformal_track}
    \sigma_h^n(z) \;=\; \int_0^{nh} \partial_s H\bigl(\Phi^\tau_H z\bigr)\,d\tau
    \;+\; O(h^r\,T) \;+\; \rho_{\mathrm{conf}}(h,T),
    \qquad nh \le T,
\end{equation}
where
\[
    \rho_{\mathrm{conf}}(h,T):=C_{\alpha,T}\,\delta_{\mathrm{BEA}}^{(1)}(h,T),
    \qquad
    \delta_{\mathrm{BEA}}^{(1)}(h,T):=
    \sup_{nh\le T}\bigl\|(\Psi_h^{(r)})^n-\Phi_{\tH_{h,q}}^{nh}\bigr\|_{C^1(U')}.
\]
Thus $\rho_{\mathrm{conf}}(h,T)=O_T(h^q)$ in the finite-smooth case and $\rho_{\mathrm{conf}}(h,T)=O(e^{-c/h})$ in the analytic case. This derivative-level observable defect is distinct from the certificate defect $\rho_{\mathrm{BEA}}$ in \cref{ass:bea}(iii).
\end{lemma}

\noindent\emph{Proof sketch.} Estimate~\eqref{eq:mod_conformal_uniform} follows by differentiating the truncated modified Hamiltonian series. In the analytic case the same conclusion can also be read from Cauchy's estimate on a slightly smaller compact set. The two conformal maps satisfy
\[
    ((\Psi_h^{(r)})^n)^*\alpha=e^{-\sigma_h^n}\alpha,
    \qquad
    (\Phi_{\tH_{h,q}}^{nh})^*\alpha=
    \exp\!\left(-\int_0^{nh}\partial_s\tH_{h,q}\circ\Phi_{\tH_{h,q}}^\tau\,d\tau\right)\alpha.
\]
The $C^1$ shadowing estimate of \cref{lem:smooth_analytic_bea} controls the difference of these pullbacks. On the fixed compact set and finite horizon their positive multipliers are bounded away from zero, so the logarithm is Lipschitz and converts this pullback error into $\rho_{\mathrm{conf}}=C_{\alpha,T}\delta_{\mathrm{BEA}}^{(1)}$. Finally, \eqref{eq:mod_conformal_uniform} and the $O(h^r)$ finite-horizon closeness of the modified and original flows give
\[
    \int_0^{nh}\partial_s\tH_{h,q}\circ\Phi_{\tH_{h,q}}^\tau\,d\tau
    =\int_0^{nh}\partial_sH\circ\Phi_H^\tau\,d\tau+O(h^rT),
\]
which proves~\eqref{eq:disc_conformal_track}. \Cref{prop:exact_conformal} sharpens the statement for the master splittings. The sum defining $\sigma_h^n$ is an \emph{exact} quadrature of $\partial_s H$ along the numerical trajectory, so the $O(h^rT)$ term is entirely quadrature-and-trajectory error and vanishes identically for constant damping. \hfill$\square$

\Crefrange{lem:contact_bch}{lem:mod_conformal} give the standard finite-horizon BEA justification for \cref{ass:bea} for the master splittings, subject to the stated regularity, compact-containment, and derivative-level shadowing hypotheses. The rate-transfer theorem itself remains conditional on \cref{ass:bea}, and \cref{fig:conformal_scaling} separately confirms~\eqref{eq:disc_conformal_track} numerically at orders $r=2$ and $r=4$.

\section{Certificate Transfer: Deferred Proofs and Remarks}
\label{app:proofs}

This appendix collects the proofs deferred from \cref{sec:main}, together with three remarks that qualify the hypotheses.

\begin{proof}[Proof of \cref{thm:rate}]
\textup{(i)} is \cref{lem:mod_decay}. \textup{(ii)} By \cref{ass:bea}\textup{(i)}, $\partial_s\tH_{h,q} = \partial_s H + \delta_h$ with $\sup_{U'}|\delta_h| \le C h^r$. Moreover, since $\tH_{h,q} - H = O(h^r)$ in $C^1(U')$, the two contact vector fields~\eqref{eq:contact_eom} are $O(h^r)$-close in $C^0(U')$, so by Gr\"onwall the flows satisfy $\sup_{t\le T}\bigl\|\Phi^t_{\tH_{h,q}}z_0 - \Phi^t_H z_0\bigr\| = O(h^r)$, and hence, using the Lipschitz bound on $\partial_s H$ from \cref{ass:reg},
\[
    \int_0^t \partial_s\tH_{h,q}\circ\Phi^\tau_{\tH_{h,q}}\,d\tau
    \;=\; \int_0^t \partial_s H\circ\Phi^\tau_H\,d\tau \;+\; O(h^r\,t).
\]
Exponentiating, $\widetilde{\Rate}_{h,q}(t) \le \Rate_H(t)\,e^{Ch^r t} \le \Rate_H(t) + C_T h^r t$, the last step using $\Rate_H(t)\le 1$ from \cref{ass:cert}\textup{(ii)} and $e^u \le 1 + u e^{u_{\max}}$ on the bounded range $u \le C h^r T$. \textup{(iii)} Split
\[
\mcE(z_n) \le \underbrace{\bigl|\mcE(z_n) - \tE_{h,q}(z_n)\bigr|}_{\le\, C_{\mathrm{mod}} h^r \ \text{(\cref{ass:bea}(i))}} + \underbrace{\bigl|\tE_{h,q}(z_n) - \tE_{h,q}(\Phi^{nh}_{\tH_{h,q}}z_0)\bigr|}_{\le\, \rho_{\mathrm{BEA}} \ \text{(\cref{ass:bea}(ii),(iii))}} + \underbrace{\tE_{h,q}(\Phi^{nh}_{\tH_{h,q}}z_0)}_{\text{bounded by (i),(ii)}},
\]
and expand $\tE_{h,q}(z_0) \le \mcE(z_0) + C_{\mathrm{mod}}h^r$ in the last term. All $O(h^r)$ constants, including $C_0$ from \textup{(i)}, are collected into $C_{\mathrm{mod},T}$. The objective bound~\eqref{eq:objective_envelope} is then immediate from~\eqref{eq:objective_comparison}.
\end{proof}

\begin{proof}[Proof of \cref{cor:aux_transfer}]
By \cref{ass:bea}\textup{(ii)}, $\|z_n - \Phi^{nh}_{\tH_{h,q}} z_0\| \le \delta_{\mathrm{BEA}}(h,T)$. By \cref{ass:bea}\textup{(i)} the vector fields of $H$ and $\tH_{h,q}$ are $O(h^r)$-close in $C^0(U')$, so Gr\"onwall gives $\|\Phi^{nh}_{\tH_{h,q}} z_0 - \Phi^{nh}_H z_0\| \le C_T' h^r$. Hence $W(z_n) \le W(\Phi^{nh}_H z_0) + L_W(C_T'h^r + \delta_{\mathrm{BEA}})$, and the continuous decay bounds the first term by $e^{-\lambda_W nh}W(z_0)$.
\end{proof}

\begin{remark}[Sign-indefinite damping]\label{rem:sign_indef}
Item (ii) of \cref{ass:cert} licenses the additive form of the envelope perturbation in~\eqref{eq:envelope} because it guarantees $\Rate_H(t)\le 1$, so the exponential factor arising in the proof of \cref{thm:rate}\textup{(ii)} can be linearized additively. If $\partial_s H$ changes sign on $U$, $\Rate_H$ can exceed one and the term $C_T h^r\,nh$ must be replaced by the multiplicative factor $e^{C h^r nh}$.
\end{remark}

\begin{remark}[On the normalization $f^\star \ge 0$]\label{rem:gauge}
\Cref{lem:mod_decay} is the reason the modified-certificate decay is a conclusion rather than a hypothesis. The multiplicative structure survives discretization because the modified Hamiltonian is itself a contact Hamiltonian. The only obstruction is the sign of the reference $H^\star$. The normalization $f^\star = 0$ deserves one caveat. For damping atoms with $\nabla_x D = 0$ (constant damping $\gamma s$), adding a constant to $f$ changes only the decoupled $s$-history, so the normalization is innocuous. For state-dependent damping $\beta(x)s$, a constant shift of $f$ alters $s(t)$ and feeds back into $\dot p$ through the contact correction $s\,\nabla\beta$, so $f^\star \ge 0$ (equivalently, knowledge of a lower reference for $f$) is a genuine hypothesis there, mirroring the reference value $f_{\mathrm{ref}}$ used by the practical schedules of \cref{app:algorithms}.
\end{remark}

\begin{remark}[Scope of the exactness]\label{rem:exact_scope}
Both hypotheses of \cref{prop:exact_conformal} are essential, and they fail in different ways. \emph{(a) Exact sub-steps.} The mechanism is multiplicativity of the pullback identity $(\Phi^\tau_{H_i})^*\alpha = \exp(-\int_0^\tau \partial_s H_i\,dt)\,\alpha$, which is available only for exact contact flows. When a sub-step is merely approximately contact, what is lost is not the accuracy of the factor but the object itself. For the re-evaluated Contact-Newton preconditioner, the pullback computed in \cref{app:catalogue} acquires an $O(\tau)$ term along $dx_j$ that is not proportional to $\alpha$, so the sub-step is not conformal for \emph{any} factor and there is nothing for multiplicativity to propagate. This is why that variant is identified as framework-motivated in \cref{tab:classical_algos}. \emph{(b) Reeb isolation.} Given exact sub-steps, multiplicativity always yields an exact cumulative factor $\exp(-\sum_i \int \partial_s H_i\,dt)$. The role of the frozen-position hypothesis is that each integrand is constant along its own sub-flow, which collapses the exponent to the sampled sum $\sum_j \beta(x_j)\tau_j$ of the nominal rate. An exact atom whose conformal rate evolves along its own flow retains the exactness but not the sampled form. For the nonlinear damping~\eqref{eq:D_nonlinear}, $\partial_s D = \gamma s$ varies during the sub-step, and the exact factor is $(1+\tfrac{\gamma}{2}s\tau)^{-2}$, the momentum multiplier in~\eqref{eq:D_nonlinear}. This is a closed-form but trajectory-weighted quadrature rather than a sample of the nominal rate. Exact conformal bookkeeping is therefore a property of splittings that isolate the Reeb derivative $R(H) = \partial_s H$ (\cref{app:contact_dynamics}) into exactly integrated, frozen-coefficient atoms, not of splitting methods per se.
\end{remark}

\section{Lyapunov/Modified-Energy Constructions}
\label{app:lyapunov}

We record the conformal Lyapunov functionals invoked in \cref{sec:main,sec:closed_form}, and prove \cref{lem:bregman_lyap}. Let $f^\star = \min_x f$ and $x^\star \in \arg\min f$.

\paragraph{Heavy ball ($D = \gamma s$, $f$ $\mu$-strongly convex).}
For the quadratic heavy-ball regime of \cref{sec:quadratic}, the bare mechanical energy $\mcE_0(x,p) = \tfrac12\|p\|^2 + (f(x) - f^\star)$ satisfies $\dot\mcE_0 = -\gamma\|p\|^2$, which is dissipative but not immediately comparable to $\mcE_0$ itself. To upgrade this to a sharp objective-level exponential certificate, one uses the usual Bregman-style quadratic adjustment
\[
    \mcE = \tfrac12\|p + \sigma(x - x^\star)\|^2 + \tau(f - f^\star),
\]
which is equivalent to $\|p\|^2 + \|x - x^\star\|^2$ on $\mu$-strongly convex quadratics for suitable positive weights~\cite{polyak1964heavyball}. At the extremal energy rate $\gamma$, however, a single scalar $\tau$ works only for a single eigenvalue: mode $\lambda_i$ forces $\sigma=\gamma/2$ and $\tau=1-\gamma^2/(4\lambda_i)$. For a general positive-definite quadratic the corresponding matrix-weighted certificate is
\[
    \mcE_\gamma(x,p)=\tfrac12\left\|p+\tfrac\gamma2(x-x^\star)\right\|^2
    +\tfrac12(x-x^\star)^T\left(A-\tfrac{\gamma^2}{4}I\right)(x-x^\star),
    \qquad \dot\mcE_\gamma=-\gamma\mcE_\gamma,
\]
which is positive definite when $\gamma<2\sqrt\mu$. Alternatively, the direct modal computation of \cref{app:quadratic_details} certifies the Hamiltonian certificate $\mcE = H - f^\star$ itself, with the explicit comparison constant~\eqref{eq:Ccmp}.

\paragraph{Nesterov polynomial regime ($\beta(t) = r/t$, $r \ge 3$, $f$ convex).}
After the autonomous lift of \cref{prop:lift}, the conformal rate is $\lambda(t) = r/t$ and the \emph{Hamiltonian} envelope is polynomial, $\Rate_H(t) = (t_0/t)^{r}$. The \emph{objective-level} certificate is supplied by the Su--Boyd--Cand\`es functional~\cite{su2016differential}
\[
    \mcE_{\mathrm{SBC}}(x,p,t) = t^2\bigl(f(x)-f^\star\bigr) + \tfrac{(r-1)^2}{2}\,\bigl\|x - x^\star + \tfrac{t}{r-1}\,p\bigr\|^2,
\]
which satisfies $\dot\mcE_{\mathrm{SBC}} \le 0$ for $r \ge 3$, hence $f(x(t))-f^\star = O(1/t^{2})$. The two envelopes should not be conflated. The contact identity delivers the $t^{-r}$ decay of the (lifted) Hamiltonian, while the certified objective rate through $\mcE_{\mathrm{SBC}}$ is $O(1/t^2)$, and faster objective rates require additional assumptions. If the positivity and comparison conditions of \cref{ass:cert} are verified for the chosen lifted initial data and compact region, \cref{thm:rate} transfers the lifted Hamiltonian envelope on $[t_0,t_0+T]$. A discrete transfer of the SBC objective certificate instead requires the auxiliary-shadowing route applied to this explicitly time-weighted functional, and yields a finite-horizon additive shadowing defect rather than a direct consequence of the conformal theorem.

\paragraph{Strongly convex ($D = 2\sqrt\mu\,s$, $f$ $\mu$-strongly convex).}
With $\lambda = 2\sqrt\mu$, a Wibisono--Wilson--Jordan-type functional~\cite{wibisono2016variational,wilson2021lyapunov}
\[
    \mcE_{\mathrm{WWJ}}(x,p) = f(x) - f^\star + \tfrac12\|p + \sqrt\mu\,(x - x^\star)\|^2
\]
satisfies $\dot\mcE_{\mathrm{WWJ}} \le -\sqrt\mu\,\mcE_{\mathrm{WWJ}}$ for general $\mu$-strongly convex, $L$-smooth $f$, certifying the rate $e^{-\sqrt\mu\,t}$. On quadratics, the sharper modal analysis of \cref{app:quadratic_details} shows the full Hamiltonian envelope rate $\gamma$ transfers to the objective with the explicit constant~\eqref{eq:Ccmp} for any $\gamma < 2\sqrt\mu$, approaching the $2\sqrt\mu$ envelope at the cost of a blowing-up constant.

\paragraph{State-dependent damping ($D = \beta(x)\,s$): proof of \cref{lem:bregman_lyap}.}
The Hamiltonian envelope is monotone whenever $\beta(x(t)) \ge 0$ along the trajectory. To upgrade this into the objective-level comparison hypothesis~\eqref{eq:objective_comparison} used in \cref{cor:state_dep}, the Polyak heavy-ball Lyapunov adjustment must be extended by an $s^2$ term to absorb the coupling $-s\,\nabla\beta(x)$ that the state-dependence introduces in the $\dot p$ equation. The $s^2$ term has two opposing effects. It supplies the dissipation $-\rho\beta s^2$ that absorbs the cross-term, and it also injects the indefinite forcing $\rho s(\tfrac12\|p\|^2 - f)$, which \emph{grows} with $\rho$. Thus $\rho$ is constrained from both sides, and the windows~\eqref{eq:lyap_windows} are exactly what makes the two constraints compatible.

\begin{proof}[Proof of \cref{lem:bregman_lyap}]
Write $y := x - x^\star$, $w := p + \sigma y$, $B := \|\nabla\beta\|_{L^\infty(U)}$, and $g := \beta(x) - \sigma$. Since $\sigma = \beta_{\min}/2$ and $\beta \ge \beta_{\min}$, we have $\beta_{\min}/2 \le g \le \beta_{\max}$. The contact equations of motion for $H$ read $\dot y = p$, $\dot p = -\nabla f - s\,\nabla\beta - \beta p$, $\dot s = \tfrac12\|p\|^2 - f - \beta s$. Differentiating $\tE = \tfrac12\|w\|^2 + f + \tfrac\rho2 s^2$ along the flow and using $-w\cdot\nabla f + \nabla f\cdot p = -\sigma\, y\cdot\nabla f$,
\begin{equation}\label{eq:lyap_master_id}
    \frac{\dx}{\dx t}\,\tE
    \;=\; -\,\sigma\, y\cdot\nabla f \;-\; g\, w\cdot p \;-\; s\, w\cdot\nabla\beta \;+\; \rho s\bigl(\tfrac12\|p\|^2 - f\bigr) \;-\; \rho\beta s^2 .
\end{equation}
We bound the five terms. (1)~Strong convexity gives $y\cdot\nabla f \ge f + \tfrac\mu2\|y\|^2$, so $-\sigma y\cdot\nabla f \le -\sigma f - \tfrac{\sigma\mu}{2}\|y\|^2$. (2)~From $p = w - \sigma y$ and AM--GM, $-g\,w\cdot p \le -\tfrac{g}{2}\|w\|^2 + \tfrac{g\sigma^2}{2}\|y\|^2$. (3)~$|s\,w\cdot\nabla\beta| \le \tfrac{B}{2}\|w\|^2 + \tfrac{B}{2}s^2$. (4)~$\tfrac12\|p\|^2 \le \|w\|^2 + \sigma^2\|y\|^2$ and $|s|\le S$ give $|\rho s(\tfrac12\|p\|^2 - f)| \le \rho S\bigl(\|w\|^2 + \sigma^2\|y\|^2 + f\bigr)$. (5)~$-\rho\beta s^2 \le -\rho\beta_{\min}s^2$. Collecting coefficients against the target $-\tfrac{\beta_{\min}}{4}\tE = -\tfrac{\beta_{\min}}{4}\bigl(\tfrac12\|w\|^2 + f + \tfrac\rho2 s^2\bigr)$, it suffices that
\begin{align*}
    \|w\|^2:&\quad -\tfrac{g}{2} + \tfrac{B}{2} + \rho S \;\le\; -\tfrac{\beta_{\min}}{8},
    & f:&\quad -\sigma + \rho S \;\le\; -\tfrac{\beta_{\min}}{4},\\
    s^2:&\quad -\rho\beta_{\min} + \tfrac{B}{2} \;\le\; -\tfrac{\beta_{\min}}{8}\,\rho,
    & \|y\|^2:&\quad -\tfrac{\sigma\mu}{2} + \tfrac{g\sigma^2}{2} + \rho S\sigma^2 \;\le\; 0,
\end{align*}
where the $\|y\|^2$ row need only be nonpositive since $\tE$ contains no explicit $\|y\|^2$ term. With $\sigma = \beta_{\min}/2$ and $\rho = \tfrac23 B/\beta_{\min}$: (H2) gives $\tfrac{B}{2} \le \tfrac{\beta_{\min}}{16}$ and (H3) gives $\rho S \le \tfrac23\cdot\tfrac{3}{32}\beta_{\min} = \tfrac{\beta_{\min}}{16}$, so the $\|w\|^2$ row is $\le -\tfrac{\beta_{\min}}{4} + \tfrac{\beta_{\min}}{16} + \tfrac{\beta_{\min}}{16} = -\tfrac{\beta_{\min}}{8}$ and the $f$ row is $\le -\tfrac{\beta_{\min}}{2} + \tfrac{\beta_{\min}}{16} \le -\tfrac{\beta_{\min}}{4}$. The $s^2$ row reads $\tfrac78\rho\beta_{\min} \ge \tfrac{B}{2}$, satisfied since $\tfrac78\cdot\tfrac23 = \tfrac{7}{12} > \tfrac12$. For the $\|y\|^2$ row, $g \le \beta_{\max}$ and $\rho S \le \beta_{\min}/16 \le \beta_{\max}/16$ give $\tfrac{g\sigma}{2} + \rho S\sigma \le \tfrac{9}{16}\sigma\beta_{\max} = \tfrac{9}{32}\beta_{\min}\beta_{\max} \le \tfrac{27}{64}\mu < \tfrac\mu2$ by (H1). This proves (ii). For (i), each summand of $\tE$ is nonnegative, and $\tE = 0$ forces $w = 0$ and $f = 0$, hence $y = 0$ by strong convexity and then $p = w - \sigma y = 0$, with $s = 0$ additionally forced when $\rho > 0$. (iii) is immediate since $f \le \tE$. When $\rho = 0$ (constant damping) the rows involving $\rho$ and $B$ vanish and the same bookkeeping returns the classical Polyak certificate \cite{polyak1964heavyball,wibisono2016variational}.
\end{proof}

\Cref{lem:bregman_lyap} supplies an explicit objective-level certificate for the state-dependent damping family. Being of non-Hamiltonian form, it transfers to the discrete iterates through the shadowing corollary (\cref{cor:aux_transfer}) under the stated strong-convexity, damping-window, and smallness assumptions on $\beta$.

\section{Quadratic Contact-Certificate Benchmark: Derivations}
\label{app:quadratic_details}

This appendix contains the derivations behind \cref{prop:quadratic} and \cref{sec:quadratic}.

\paragraph{Modal reduction and explicit solution.}
With $A = Q\Lambda Q^T$, $y = Q^Tx$, $q = Q^Tp$, the flow separates into scalar damped oscillators $\ddot y_i + \gamma\dot y_i + \lambda_i y_i = 0$. For $\gamma < 2\sqrt{\lambda_i}$ set $\omega_i := \sqrt{\lambda_i - \gamma^2/4}$; with $p_0 = 0$,
\begin{equation}\label{eq:modal_solution}
    y_i(t) = e^{-\gamma t/2}\,y_{0,i}\,g_i(t),
    \qquad
    g_i(t) := \cos\omega_i t + \frac{\gamma}{2\omega_i}\sin\omega_i t.
\end{equation}

\paragraph{Certificate and comparison constant.}
With $p_0=0$, $s_0=0$ we have $H(0) = f(x_0)$ and, by the identity $\dot H = -\gamma H$, $\mcE(t) = H(t) = e^{-\gamma t} f(x_0) \ge 0$, so \cref{ass:cert}(i)--(ii) hold, with equality in the induced decay~\eqref{eq:E_decay_assumption}. For the comparison, harmonic addition gives $\sup_t g_i(t)^2 = 1 + \gamma^2/(4\omega_i^2) = 4\lambda_i/(4\lambda_i - \gamma^2)$, so
\begin{equation}\label{eq:Ccmp}
    f(x(t)) = \sum_i \tfrac12\lambda_i y_i(t)^2
    \;\le\; e^{-\gamma t}\sum_i \tfrac12 \lambda_i y_{0,i}^2\cdot\frac{4\lambda_i}{4\lambda_i-\gamma^2}
    \;\le\; \frac{4\mu}{4\mu-\gamma^2}\, e^{-\gamma t} f(x_0)
    \;=\; C_{\mathrm{cmp}}\,\mcE(t),
\end{equation}
using that $\lambda \mapsto 4\lambda/(4\lambda-\gamma^2)$ is decreasing, so the maximum is attained at the actual smallest eigenvalue $\mu$. Initial data in that eigenspace attain the constant, hence it is sharp. Since $g_i^2$ exceeds $1$ on a set of positive measure whenever $\gamma>0$, $K+D = \mcE - (V - f^\star)$ is transiently negative. The pointwise sufficient condition $K+D\ge0$ fails while the comparison holds. For an overdamped mode ($\gamma > 2\sqrt{\lambda_i}$, $\eta_i := \sqrt{\gamma^2/4 - \lambda_i}$), the slow branch decays like $e^{-(\gamma/2-\eta_i)t}$, so $(V-f^\star)/\mcE \sim e^{2\eta_i t} \to \infty$ and no horizon-uniform constant exists. At critical damping the secular factor $t^2 e^{-\gamma t}$ likewise escapes any constant. This proves \cref{prop:quadratic}(i). For any fixed finite horizon and nontrivial initial datum, continuity and the positive lower bound $\mcE(t)\ge e^{-\gamma T}\mcE(0)$ extend the comparison to a sufficiently small compact tube around the trajectory, with constant $C_{\mathrm{cmp}}+\varepsilon$. Thus the neighborhood form of \cref{ass:cert} used for discrete transfer follows for sufficiently small $h$, but its sharp constant is the continuous one above. The action variable is recovered algebraically from the envelope, $s(t) = \gamma^{-1}(e^{-\gamma t}H(0) - \tfrac12\|p(t)\|^2 - f(x(t)))$, which is~\eqref{eq:s_closed_form}.

\paragraph{The discrete Strang mode map.}
Per mode, the sub-flows act on $(y_i, q_i)$ linearly: $\Phi_K^h = \begin{psmallmatrix}1 & h\\ 0 & 1\end{psmallmatrix}$, $\Phi_V^{h/2} = \begin{psmallmatrix}1 & 0\\ -v_i & 1\end{psmallmatrix}$ with $v_i = \lambda_i h/2$, and $\Phi_D^{h/2} = \mathrm{diag}(1, \delta)$ with $\delta = e^{-\gamma h/2}$. The Strang composition~\eqref{eq:strang} is the matrix product $M_i(h) = \Phi_D^{h/2}\Phi_V^{h/2}\Phi_K^{h}\Phi_V^{h/2}\Phi_D^{h/2}$. Its projected $(y_i,q_i)$ dynamics are the kick--drift--kick adjoint of the dissipative drift--kick--drift leapfrog analyzed in~\cite{francca2020conformal,francca2021dissipative}, and the two orderings have the same characteristic polynomial. We include the calculation to connect that known spectrum to the augmented contact certificate. The inner Verlet block is
\[
    C_i := \begin{pmatrix}1 & 0\\ -v_i & 1\end{pmatrix}
    \begin{pmatrix}1 & h\\ 0 & 1\end{pmatrix}
    \begin{pmatrix}1 & 0\\ -v_i & 1\end{pmatrix}
    = \begin{pmatrix} 1 - hv_i & h \\ -v_i(2-hv_i) & 1 - hv_i\end{pmatrix},
    \qquad \det C_i = 1,
\]
and sandwiching with the damping factors gives
\[
\begin{aligned}
    M_i(h) &= \begin{pmatrix} 1 - h v_i & \delta h \\ -\delta v_i(2-hv_i) & \delta^2(1-hv_i)\end{pmatrix},\\
    \det M_i(h) &= \delta^2 = e^{-\gamma h},
    &\operatorname{tr} M_i(h) &= (1-hv_i)(1+\delta^2).
\end{aligned}
\]
The determinant identity holds because the Verlet block is symplectic and each damping half-step scales the momentum by exactly $\delta$. This is \cref{prop:exact_conformal} in matrix form. The eigenvalues are complex iff $\operatorname{tr}^2 < 4\det$, i.e.
\[
    \bigl(1 - \tfrac{\lambda_i h^2}{2}\bigr)^2 (1+\delta^2)^2 < 4\delta^2
    \iff
    \Bigl|1 - \frac{\lambda_i h^2}{2}\Bigr| < \frac{2\delta}{1+\delta^2} = \operatorname{sech}\!\Bigl(\frac{\gamma h}{2}\Bigr),
\]
and in that window $|\mathrm{eig}\,M_i| = \sqrt{\det M_i} = e^{-\gamma h/2}$ exactly. At $\gamma = 0$ the window is the St\"ormer--Verlet stability interval $0 < \lambda_i h^2 < 4$~\cite{hairer2006geometric}. Outside the window the two eigenvalues are real and distinct, so they no longer share the common modulus $e^{-\gamma h/2}$ but separate into a fast and a slow branch, with the product of their moduli still fixed at $\det M_i = e^{-\gamma h}$. For the phase, $\cos\theta_i(h) = \operatorname{tr}/(2\sqrt{\det}) = (1-\lambda_i h^2/2)\cosh(\gamma h/2) = 1 - (\tfrac{\lambda_i}{2} - \tfrac{\gamma^2}{8})h^2 + O(h^4)$, which agrees with $\cos(\omega_i h)$ through $O(h^2)$, giving $\theta_i(h) = \omega_i h + O(h^3)$. Exact modulus does not imply exact amplitude for physical initial data: for $q_{0,i}=0$,
\[
    y_{n,i}=e^{-\gamma nh/2}y_{0,i}\left(\cos(n\theta_i)+
    \frac{(1-\lambda_i h^2/2)\sinh(\gamma h/2)}{\sin\theta_i}\sin(n\theta_i)\right),
\]
whose sine coefficient generally differs from the exact $\gamma/(2\omega_i)$. Thus the scheme has an exact spectral envelope but retains phase and modal-shape errors. This proves \cref{prop:quadratic}(ii). \Cref{thm:rate} then guarantees, and \cref{fig:quad_envelope} confirms, that the full Hamiltonian certificate along the iterates (including the $\gamma s_n$ term) tracks the envelope $e^{-\gamma nh}$ with a bounded $O(h^2)$ relative ripple over the horizon.

\section{Autonomous Lift of Time-Dependent Damping}
\label{app:lift}

This appendix records the autonomous lift invoked in \cref{sec:recovery} for the time-dependent damping models of \cref{tab:classical_algos_td}.

\begin{prop}[Autonomous lift of time-dependent damping]\label{prop:lift}
Let $\beta \in C^k([t_0,\infty))$ with $t_0>0$, and consider the nonautonomous dynamics $\ddot x + \beta(t)\dot x + \nabla f(x) = 0$. On $J^1(\R^{n+1})$ with base coordinates $(x,\theta)$, conjugate momenta $(p,\pi)$, contact variable $s$, and contact form $\alpha = ds - p^T dx - \pi\, d\theta$, define the lifted contact Hamiltonian
\begin{equation}\label{eq:lifted_H}
    \bar H(x,\theta,p,\pi,s) \;=\; \tfrac12\|p\|^2 + f(x) + \pi + \beta(\theta)\,s.
\end{equation}
Then along the contact flow of $\bar H$: (i) $\theta(t) = \theta_0 + t$, so choosing $\theta_0 = t_0$ identifies $\theta$ with time; (ii) the $(x,p)$ dynamics reproduce the nonautonomous system; (iii) the conformal rate is $\partial_s \bar H = \beta(\theta)$, i.e.\ the time-dependent damping law. Moreover $\bar H$ is again of master form: $\tfrac12\|p\|^2$ and $f(x)$ are strict atoms, and $\pi + \beta(\theta)s$ is affine in the momenta, hence a prolonged atom, so the geometric splitting framework of \cref{sec:subflows} applies on compact subsets of $\{\theta \ge t_0\}$. The certificate hypotheses of \cref{ass:cert} must still be verified for the lifted Hamiltonian. For the Nesterov choice $\beta(t) = r/t$, the singularity at $t=0$ is excluded by the restriction $\theta \ge t_0 > 0$, and all compactness hypotheses of \cref{sec:main} are imposed on $[t_0, t_0+T]$.
\end{prop}

\begin{proof}
On $J^1(\R^{n+1})$ with $\alpha = ds - p^Tdx - \pi\,d\theta$, the contact equations~\eqref{eq:contact_eom} for $\bar H = \tfrac12\|p\|^2 + f(x) + \pi + \beta(\theta)s$ read
\[
    \dot x = \nabla_p\bar H = p, \qquad
    \dot\theta = \partial_\pi \bar H = 1, \qquad
    \dot p = -\nabla_x \bar H - p\,\partial_s\bar H = -\nabla f(x) - \beta(\theta)\,p,
\]
\[
    \dot\pi = -\partial_\theta \bar H - \pi\,\partial_s\bar H = -\beta'(\theta)\,s - \beta(\theta)\,\pi, \qquad
    \dot s = p^T\nabla_p\bar H + \pi\,\partial_\pi\bar H - \bar H = \|p\|^2 + \pi - \bar H.
\]
From $\dot\theta = 1$, $\theta(t) = \theta_0 + t$. Substituting into the $(x,p)$ equations reproduces $\ddot x + \beta(\theta_0+t)\dot x + \nabla f(x) = 0$, which is the nonautonomous dynamics with time origin $\theta_0$. The conformal rate is $\partial_s\bar H = \beta(\theta)$ as claimed, and the identity $\dot{\bar H} = -\bar H\,\partial_s\bar H$ holds verbatim for the lifted system. This supplies the lifted contact identity, but not the optimization certificate automatically. Because $\bar H$ contains the free momentum $\pi$, the quantity $\bar H-f^\star$ need not be nonnegative or control $f-f^\star$. Those parts of \cref{ass:cert} require a separate invariant-region or auxiliary-certificate argument. Structurally, $\tfrac12\|p\|^2$ and $f(x)$ are independent of $s$ (strict atoms), while $\pi + \beta(\theta)s$ is affine in the momenta $(p,\pi)$ with coefficients depending on the base coordinate $\theta$ and on $s$, hence a prolonged atom in the sense of~\cite{kevrekidis2026local}. Since $\theta$ \emph{moves} during this sub-flow ($\dot\theta = \partial_\pi = 1$), the exact sub-flow is \emph{not} a frozen-coefficient rescaling. It translates $\theta_0 \mapsto \theta_0 + \tau$ and rescales $p$ and $s$ by the integrated factor
\[
    \exp\Bigl(-\int_0^\tau \beta(\theta_0 + t)\,dt\Bigr)
    \qquad\Bigl(\text{for } \beta(t) = r/t:\ \bigl(\tfrac{\theta_0}{\theta_0+\tau}\bigr)^{r}\Bigr),
\]
with $\pi$ obtained by variation of constants from the linear equation $\dot\pi = -\beta'(\theta)s - \beta(\theta)\pi$ along the known $s(t)$. All steps are in closed form whenever $\int\beta$ is. For $\beta(t)=r/t$ the data are smooth (indeed analytic) on $\{\theta \ge t_0\}$ for any $t_0>0$, so \cref{ass:reg} can be satisfied on compact subsets away from the singularity, with horizon $[t_0, t_0+T]$.
\end{proof}

\section{List of Algorithms}
\label{app:algorithms}

We list the algorithms used in the experiments of \cref{sec:computation}. A ``C-" prefix denotes a contact structure-preserving discretization (Strang) of the continuous dynamics. Unlike \cref{tab:classical_algos}, which records abstract contact-Hamiltonian templates, the table below records the parameterization used for the reported experiments. For the objective-adaptive rows, write $\rho(x)$ for the clipped normalized objective gap $(f(x)-f_{\mathrm{ref}})/f_{\mathrm{scale}}$.

\begin{table}[t]
\centering
\footnotesize
\renewcommand{\arraystretch}{1.25}
\begin{tabular}{l l l >{\raggedright\arraybackslash}p{0.47\textwidth}}
\hline
Algorithm & $K(x,p)$ & $D(x,p,s)$ & Parameters\\
C-HB & $\frac{1}{2}\|p\|^2$ & $\beta s$ & $\tau$ and \texttt{beta\_tau} (the product $\beta\tau$) \\
C-NAG & $\frac{1}{2}\|p\|^2$ & $\beta_k s$ & $\tau$, $r$, and $k_0$, with $\beta_k = r/((k+k_0)\tau)$ \\
C-Adapt & $\frac{1}{2}\|p\|^2$ & $\beta(x) s$ & $\tau$, \texttt{beta\_floor}, \texttt{beta\_gain}, \texttt{f\_scale}, \texttt{f\_ref}, and \texttt{power}, with $\beta(x)=\beta_{\mathrm{floor}}+\beta_{\mathrm{gain}}\rho(x)^q$ \\
C-InvGrad & $\frac{1}{2}\|p\|^2$ & $\beta(x) s$ & $\tau$, \texttt{beta\_floor}, \texttt{beta\_gain}, \texttt{c}, \texttt{f\_scale}, \texttt{f\_ref}, and \texttt{power}, with $\beta(x)=\beta_{\mathrm{floor}}+\beta_{\mathrm{gain}}/\bigl(1+c\rho(x)^q\bigr)$ \\
C-Adam & $\frac{1}{2} p^\top\!\mathrm{diag}(\sqrt{\hat v}+\epsilon)^{-1}p$ & $\beta_{\mathrm{eff}}(x,p,k)s$ & $\tau$, \texttt{beta\_contact}, $\beta_2$, $\epsilon$, and \texttt{splitting}; scheduled damping adds \texttt{beta\_min}; adaptive damping uses floor/gain/power/inverse/clip parameters \\
C-SGD & $\frac{1}{2}\|p\|^2$ & $\beta_{\mathrm{eff}}(k)s$ & $\tau$, \texttt{beta\_contact}, and \texttt{splitting}; the Nesterov variant uses $\beta_k = 3/(k\tau+1)$ \\
\hline
\end{tabular}
\caption{Experiment-facing contact algorithm families used in the deterministic and deep-learning sections. The iteration-scheduled rows (C-NAG, scheduled C-SGD and C-Adam variants) are nonautonomous in the sense of \cref{tab:classical_algos_td}. The deterministic exact-gradient case is covered by the theory only through the lift of \cref{prop:lift}, while stochastic and adaptive-moment variants are reported as design-template stress tests rather than certified deployments.}
\label{tab:implemented_algos}
\end{table}

\paragraph{Hyperparameter Tuning.} For the deterministic benchmarks, we tune each method over problem-specific grids of its native step-size and damping or momentum parameters, first using shorter screening runs and then reranking the strongest candidates on longer horizons before the reported full-budget comparisons. For the deep-learning benchmarks, we use lightweight sweep runs to choose baseline learning rates and contact step-size or damping settings, and then evaluate the selected configurations in full-budget multi-seed runs, reporting aggregate statistics in the main text and appendix figures.

\paragraph{Reproducibility.} CIFAR-10 statistics aggregate three seeds. The deterministic benchmarks report full-budget runs at the tuned configurations, with sensitivity to step size, damping parameters, and initial conditions quantified separately in \cref{app:extended_numerics}. The reported deep-learning numbers should be read as aggregate performance summaries rather than uncertainty-resolved comparisons. The theory figures of \cref{sec:closed_form} are verified against direct numerical integration, with the fitted convergence orders checked against their predicted values.

\section{Extended Numerical Results}
\label{app:extended_numerics}

\paragraph{Contact Hamiltonian Ablation.} For the deterministic benchmarks of \cref{sec:computation}, we ablate the choice of Hamiltonian by comparing the performance of the main proposed algorithms to variants in which either the kinetic $K$ or dissipation $D$ term is removed. We compare to a baseline of a tuned Nesterov's accelerated gradient (NAG) method. The results are presented in \cref{fig:hamiltonian_ablations}. Overall, the full algorithm significantly outperforms the ablated variants, demonstrating the importance of \textbf{both} the kinetic and dissipation terms for achieving improved performance.

\begin{figure}
    \begin{subfigure}[t]{0.48\textwidth}
        \centering
        \includegraphics[width=\linewidth]{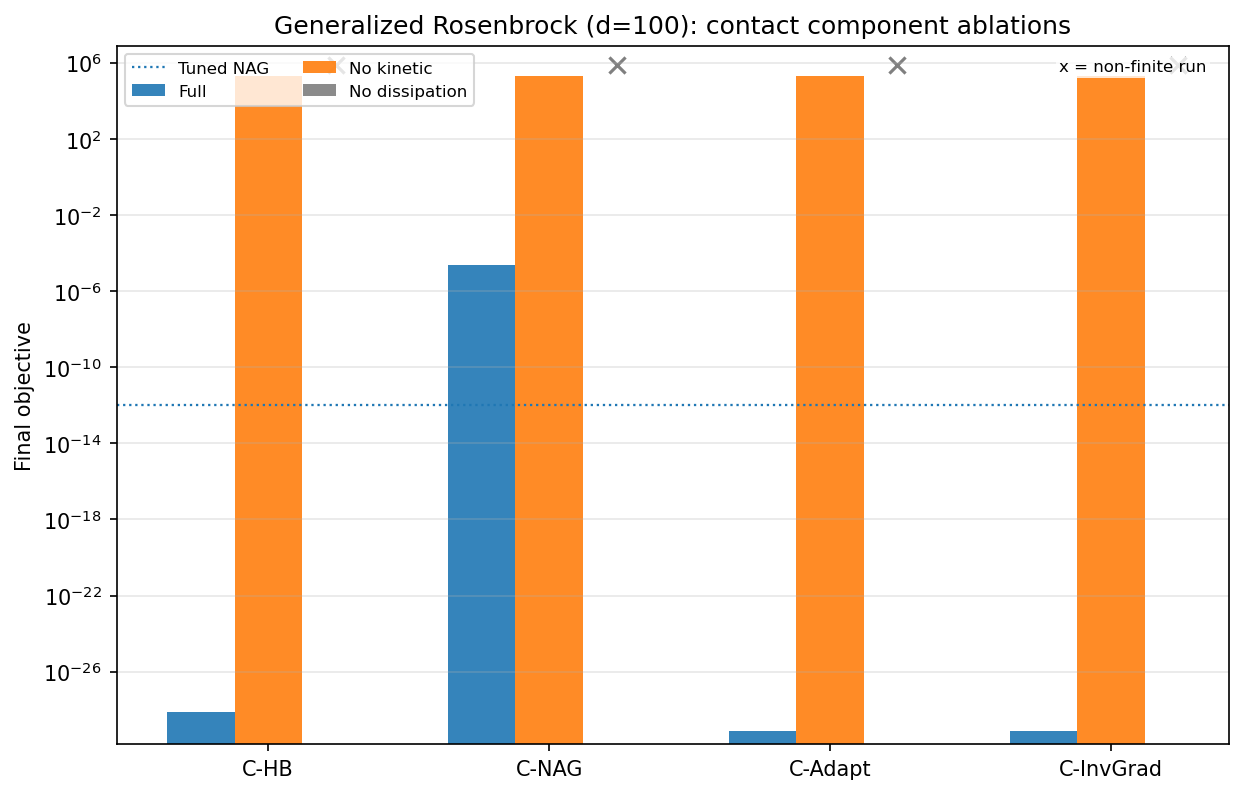}
        \caption{Rosenbrock Hamiltonian ablation.}
    \end{subfigure}\hfill
    \begin{subfigure}[t]{0.48\textwidth}
        \centering
        \includegraphics[width=\linewidth]{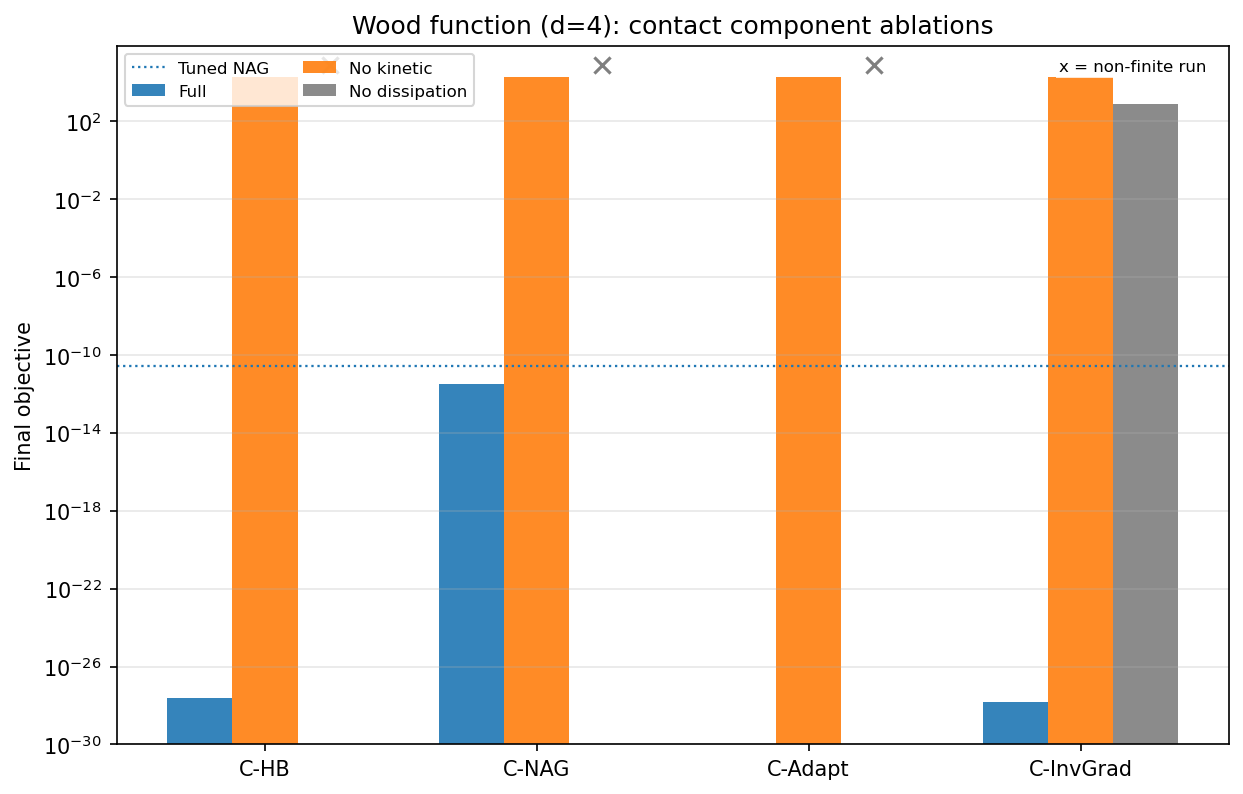}
        \caption{Wood Hamiltonian ablation.}
    \end{subfigure}
    \caption{Ablation runs for the deterministic benchmarks of \cref{sec:computation}. \textbf{Left:} Rosenbrock in $100$ dimensions. \textbf{Right:} Wood in $4$ dimensions. Blue columns represent the full proposed algorithm, while orange and gray represent variants where the kinetic or dissipation term is respectively removed.}
    \label{fig:hamiltonian_ablations}
\end{figure}

\paragraph{Robustness to Parameter Choice.} One of the benefits of geometric integration, and here specifically the contact Hamiltonian formalism, is that the structure-preservation properties of the integrator can lead to improved robustness to the choice of hyperparameters. To demonstrate this, we perform a robustness study in which we vary the main hyperparameters of the proposed contact algorithms (with the primary being the step size $\tau$ of the splitting) and observe the resulting performance on the deterministic benchmarks of \cref{sec:computation}. We find that the performance of the proposed contact algorithms is robust to a wide range of hyperparameter choices, remaining below the tuned-NAG baseline across a wide range of step sizes and other hyperparameters, demonstrating the potential benefits of the contact Hamiltonian formalism for designing optimization algorithms that are robust to hyperparameter choice.

\begin{figure}
    \begin{subfigure}[t]{0.99\textwidth}
        \centering
        \includegraphics[width=\linewidth]{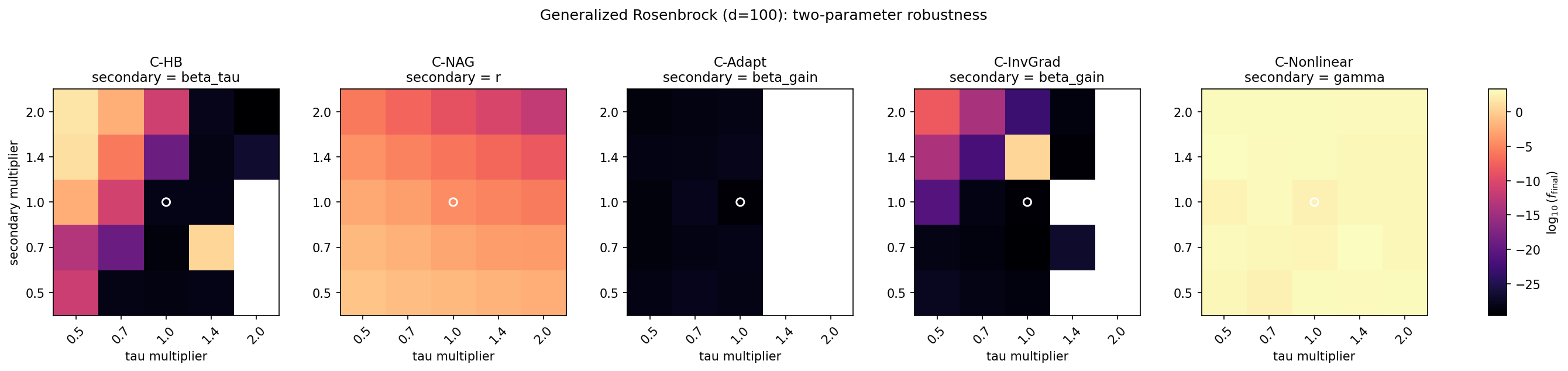}
        \caption{Rosenbrock robustness to parameter choice.}
    \end{subfigure}\hfill\\
    \begin{subfigure}[t]{0.99\textwidth}
        \centering
        \includegraphics[width=\linewidth]{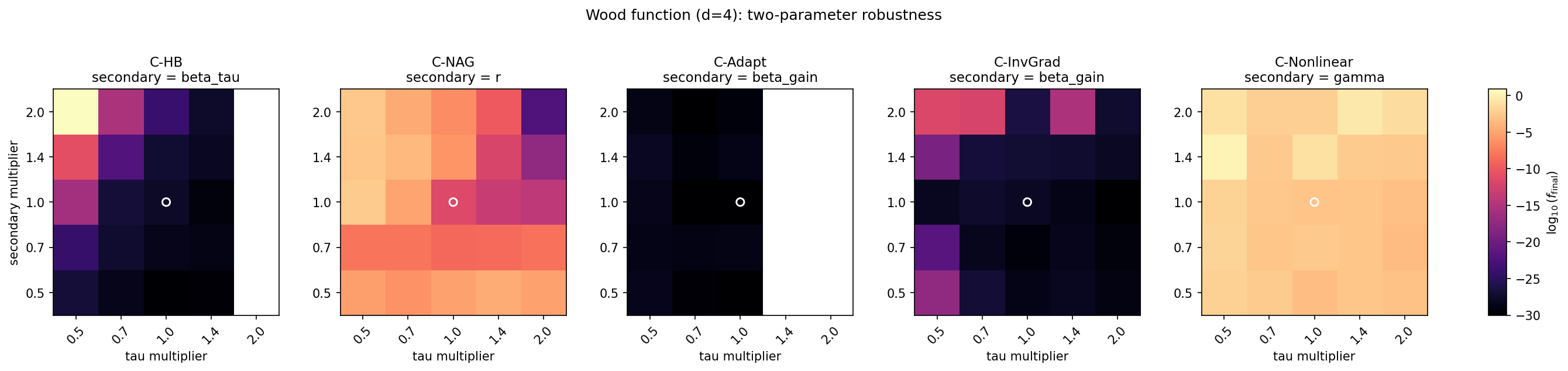}
        \caption{Wood robustness to parameter choice.}
    \end{subfigure}
    \caption{Robustness runs for the deterministic benchmarks of \cref{sec:computation}. \textbf{Top:} Rosenbrock in $100$ dimensions. \textbf{Bottom:} Wood in $4$ dimensions. Blue columns represent the full proposed algorithm, while orange and gray represent variants where the main hyperparameters are varied.}
    \label{fig:parameter_robustness}
\end{figure}

\paragraph{Robustness to Initial Conditions.} One of the advantages of state-dependent damping is that it can adapt to the local geometry of the objective function, which can lead to improved robustness to the choice of initial conditions, even if the method is tuned at a single deterministic initial condition. To demonstrate this, we perform a robustness study in which we vary the initial conditions over $100$ randomly chosen points in the $100$-dimensional Rosenbrock example, after tuning the algorithms at a single initial condition. In \cref{fig:initial_condition_robustness} we observe that the state-dependent damping algorithms remain competitive with the tuned NAG baseline across a wide range of initial conditions, obtaining better accuracy over the majority of the sampled points, with median error below $10^{-25}$ (contact algorithms) vs. approximately $10^{-12}$ (NAG), demonstrating the potential benefits of state-dependent damping for optimization. In contrast, the geometric relativistic gradient descent (RGD) method retains the edge over NAG only on a subset of the initial conditions, with significantly worse performance on the remaining initial conditions, demonstrating that the state-dependence of the damping is crucial for achieving improved robustness to initial condition choice.

\begin{figure}
    \centering
    \includegraphics[width=0.98\textwidth]{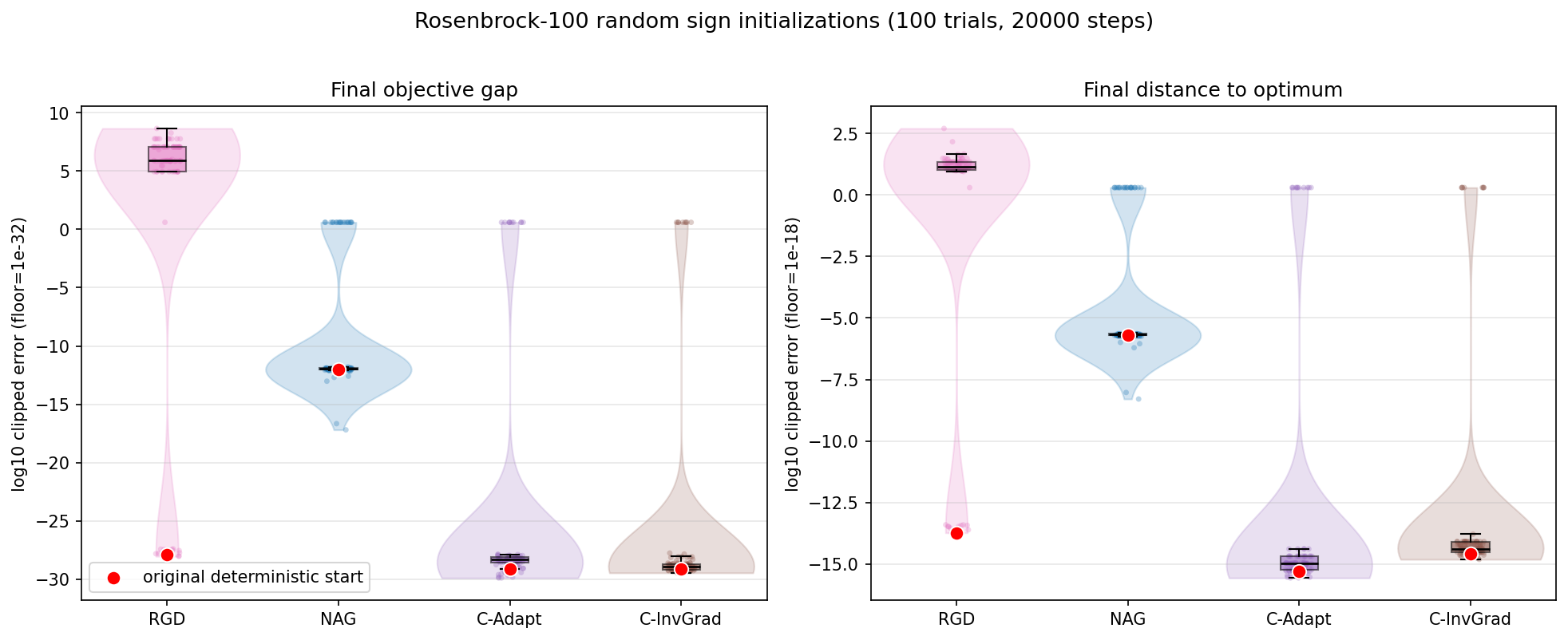}
    \caption{Rosenbrock robustness to initial condition choice. All algorithms are tuned at a single initial condition (red point), but contact algorithms remain competitive with the tuned NAG baseline across a wide range of initial conditions, demonstrating the potential benefits of state-dependent damping for optimization.}
    \label{fig:initial_condition_robustness}
\end{figure}

\end{document}

%% file: preamble.tex
\usepackage[margin=1.0in]{geometry}
\usepackage[numbers,square,comma,sort&compress]{natbib}

\usepackage{graphicx}
\usepackage{caption}
\usepackage{subcaption}
\usepackage{tikz}
\usepackage{tikz-cd}

\usepackage{physics}
\usepackage{amsmath}
\usepackage{amsthm}
\usepackage{amstext}
\usepackage{amssymb}
\usepackage{thmtools}
\usepackage{bm}

\newtheorem{theorem}{Theorem}[section]
\newtheorem{corollary}[theorem]{Corollary}
\newtheorem{lemma}[theorem]{Lemma}
\newtheorem{prop}[theorem]{Proposition}

\theoremstyle{definition}
\newtheorem{defn}{Definition}[section]
\newtheorem{exmp}{Example}[section]

\theoremstyle{definition}
\newtheorem{assumption}{Assumption}

\theoremstyle{remark}
\newtheorem{remark}[theorem]{Remark}

\usepackage[english]{babel}

\usepackage{pdfpages}

\usepackage{dsfont}
\usepackage{xcolor}
\definecolor{amber}{rgb}{1.0, 0.6, 0.0}
\usepackage{textcomp}
\usepackage{listings}
\usepackage{relsize}
\usepackage{multicol}
\usepackage{hyperref}
\hypersetup{
	colorlinks=true,
	linkcolor=black,
	filecolor=black,
	urlcolor=black,
	citecolor=blue,
}

\AtBeginDocument{%
}
\usepackage[capitalise]{cleveref}
\crefname{figure}{Fig.}{Fig.}
\crefname{table}{Table}{Tables}
\crefname{equation}{Eqn.}{Eqns.}
\crefname{algocf}{Algorithm}{Algorithms}
\crefname{exmp}{Example}{Ex.}

\crefname{lemma}{Lemma}{Lemmas}
\crefname{prop}{Proposition}{Propositions}
\crefname{corollary}{Corollary}{Cor.}
\crefname{defn}{Definition}{Definitions}
\crefname{assumption}{Assumption}{Assumptions}

\usepackage{url}

\usepackage{pifont}

\usepackage{bbm}
\usepackage{mathtools}
\usepackage{hhline}
\usepackage{tabularx}
\usepackage{booktabs}
\usepackage{nicefrac}
\usepackage{microtype}

\newcommand{\dx}{\mathrm{d}}

\newcommand{\R}{\mathbb{R}}

\newcommand{\mcM}{\mathcal{M}}

%% file: macros.tex
\newcommand{\mcE}{\mathcal{E}}

\newcommand{\tH}{\widetilde{H}}
\newcommand{\tlam}{\widetilde{\lambda}}
\newcommand{\tE}{\widetilde{\mcE}}
\newcommand{\Rate}{R}